\documentclass[review]{elsarticle}

\usepackage{amssymb,amsmath}
\usepackage{amsthm}
\usepackage{esint}
\usepackage{xcolor}

\newcommand{\vast}{\bBigg@{4}}
\newcommand{\Vast}{\bBigg@{5}}
\makeatother
\newcommand{\eps}{\varepsilon}

\newcommand{\N}{{\mathbb{N}}}

\newcommand{\R}{\mathbb{R}}
\newcommand{\p}{\partial}

\newcommand{\loc}{\textnormal{loc}}
\newcommand{\norm}[2][]{\left\|{#2}\right\|_{#1}}

\newcommand{\set}[1]{\left\{#1\right\}}

\newcommand{\textas}{\text{ as }}
\newcommand{\texton}{\text{ on }}

\newcommand{\textin}{\text{ in }}

\newcommand{\textfor}{\text{ for }}
\newcommand{\textforall}{\text{ for all }}

\newcommand{\Pg}{P_{\gamma}^{g}}
\newcommand{\Pgo}{P_{\gamma}^{g_0}}
\def\uinf{u_{\infty}}

\DeclareMathOperator{\vol}{vol}
\DeclareMathOperator{\volbar}{\overline{vol}}
\DeclareMathOperator{\const}{const}
\DeclareMathOperator{\Div}{div}

\theoremstyle{plain}
\newtheorem{thm}{Theorem}[section]
\newtheorem{lem}[thm]{Lemma}
\newtheorem{cor}[thm]{Corollary}
\newtheorem{prop}[thm]{Proposition}

\newtheorem{conj}[thm]{Conjecture}

\theoremstyle{definition}

\theoremstyle{remark}
\newtheorem{remark}[thm]{Remark}
\newcommand{\bremark}{\begin{remark} \em}
\newcommand{\eremark}{\end{remark} }
\usepackage{color}

\hbadness=\maxdimen

\begin{document}

\begin{frontmatter}

\title{Convergence of the fractional Yamabe flow for a class of initial data}

\author[H. Chan]{Hardy Chan}
\ead{hardy@math.ubc.ca}
\author[Y. Sire]{Yannick Sire\corref{cor1}}
\ead{sire@math.jhu.edu}
\author[Y. Sire]{Liming Sun}
\ead{sunlimingbit@gmail.com}
\address[H. Chan]{Department of Mathematics, University of British Columbia,\\ Vancouver, B.C., Canada, V6T 1Z2}
\address[Y. Sire]{Department of Mathematics, Johns Hopkins University,\\
3400 N. Charles St., Baltimore, MD 21218, USA}

\cortext[cor1]{Corresponding author}

\begin{abstract}
This work is a follow-up on the work of the second author with P. Daskalopoulos and J.L. V\'{a}zquez \cite{Daskalopoulos2017}. In this latter work, we introduced the Yamabe flow associated to the so-called fractional curvature and prove some existence result of mild (semi-group) solutions. In the present work, we continue this study by proving that for some class of data one can prove actually convergence of the flow in a more general context. We build on the approach in \cite{struwe} as simplified in the book of M. Struwe \cite{bookStruwe}.
\end{abstract}

\end{frontmatter}

\tableofcontents
\section{Introduction}

The resolution of the Yamabe problem, i.e. finding a metric in a given conformal of a closed manifold with constant scalar curvature has been a landmark in geometric analysis after the series of works \cite{yamabe,trudinger,aubin,schoen}. Later a parabolic proof of the previous elliptic results, was somehow desirable and in his seminal paper Hamilton \cite{hamilton} introduced the so-called Yamabe flow. Given a compact Riemannian manifold $(M, g_0)$ of dimension $n \geq 2$, Hamilton introduced in \cite{hamilton} the following evolution for a metric $g(t)$
\begin{equation}
\left\{
\begin{array}{l}
\partial_t g(t) = -\Big ( \text{Scal}_{g(t)}-\text{scal}_{g(t)}\Big )g(t)\\
g(0)=g_0,
\end{array}
\right .
\end{equation}
where $\text{Scal}_{g(t)}$ is the scalar curvature of $g(t)$ and
$$
\text{scal}_{g(t)}=\text{vol}_{g(t)}(M)^{-1} \int_M \text{Scal}_{g(t)}\,d\text{vol}_{g(t)}.
$$
This gave rise to an extensive literature, see e.g. \cite{chow,ye,struwe,brendle1,brendle2}.

On the other hand, in a seminal paper \cite{GZ} Graham and Zworski constructed for every $\gamma \in (0,n/2)$ a conformally covariant operator $P^g_\gamma$ on the conformal infinity of a Poincar\'e--Einstein manifold. These operators appear to be the higher-order generalizations of the conformal Laplacian. They coincide with the GJMS operators of \cite{GJMS} for suitable integer values of $\gamma$. This paved the way to define an interpolated quantity $R^g_\gamma$ for each $\gamma \in (0,n/2)$, which is just the scalar curvature for $\gamma=1$, and the $Q$-curvature for $\gamma=2$. This new notion of curvature has been investigated in \cite{QR,QG,CG, GMS, Kim2018} and is called the fractional curvature. Unfortunately, this notion of curvature (except in the case $\gamma=\frac12$ (see \cite{CG})), at the present knowledge, does not carry any clear geometric meaning. Nonetheless, from the analytical point of view, it interpolates between several well-known geometric quantities and one can hope that their investigations will shed some light on these matters.

In the aforementioned series of papers, all the technqiues used in studying the so-called fractional Yamabe problem are of elliptic nature. The aim of the present article is to develop a parabolic theory. The paper is twofold. We first collect all the necessary tools to deal with this new fractional flow. Then we prove convergence for certain class of initial data.

We now introduce the flow under study. On a compact Poincar\'{e}--Einstein manifold $(M,g_0)$ let $P_\gamma^g$, where $\gamma\in(0,1)\subset(0,\frac{n}{2})$ be the conformal fractional Laplacian satisfying
\begin{equation}\label{eq:conformalP}
P_\gamma^{g_0}(uf)=u^{\frac{n+2\gamma}{n-2\gamma}}P_\gamma^g(f)
    \quad\textforall{f}\in{C}^{\infty}(M),
\end{equation}
under the conformal change
\begin{equation}\label{eq:conformalg}
g=u^{\frac{4}{n-2\gamma}}g_0.
\end{equation}
In particular on $(\R^n,|dx|^2)$ we have $P_\gamma^{|dx|^2}=(-\Delta_{\R^n})^\gamma$.

The volume element on $(M,g_0)$ is denoted by $d\mu_0$. By replacing $g_0$ by its constant multiple we may assume the $(M,g_0)$ has unit volume, $\mu_0(M)=1$. With a conformal metric \eqref{eq:conformalg} we write
\[d\mu=d\mu_g=u^{\frac{2n}{n-2\gamma}}\,d\mu_0.\]

Let $R=R_\gamma^g=P_\gamma^g(1)=u^{-\frac{n+2\gamma}{n-2\gamma}}P_\gamma^{g_0}(u)$ be the fractional curvature. As previously mentioned, this is the scalar curvature when $\gamma=1$ and the $Q$-curvature when $\gamma=2$. Its average is denoted by
\[s=s_\gamma^g=\int_{M}R_\gamma^g\,d\mu.\]
Consider the 
volume-preserving fractional (note the suppressed $\gamma$) Yamabe flow
\[\begin{cases}
\frac{n-2\gamma}{4}\p_tg=(s-R)g,\\
g(0)=g_0,
\end{cases}\]
i.e.
\begin{equation}\label{eq:flow}\begin{cases}
\p_t{u}=(s-R)u,\\
u(0)=1.
\end{cases}\end{equation}
This new geometrical problem has been firstly introduced by Jin and Xiong in \cite{JX} where the authors investigate the flow on the sphere $M=\mathbb S^n$ with the round metric, the conformally flat case. Only in this context was the flow actually introduced, but the generalization on any compact manifold $M$ is straightforward and has been done in \cite{Daskalopoulos2017}. That the flow preserves the volume in time is a rather important property for the global existence.

Depending on the need, the flow \eqref{eq:flow} is sometimes alternatively expressed as a fast diffusion fractional equation, namely
\[\begin{split}
\frac{n-2\gamma}{n+2\gamma}\p_t\left(u^{\frac{n+2\gamma}{n-2\gamma}}\right)
    &=-P_\gamma^{g_0}(u)+s_\gamma^{g}u^{\frac{n+2\gamma}{n-2\gamma}}. \end{split}\]

It is convenient define the Yamabe functional
\begin{equation}\label{eq:E}
E(u)
=\frac{\displaystyle\int_{M}u\Pgo{u}\,d\mu_{0}}
    {\left(\displaystyle\int_{M}u^{\frac{2n}{n-2\gamma}}d\mu_{0}
        \right)^{\frac{n-2\gamma}{n}}},
\end{equation}
as it appears naturally in the variational formulation throughout the paper. Then the Yamabe constant for the class $[g]$ containing $g_0$ is given by
\begin{equation}\label{eq:Ygamma}
Y_\gamma(M,[g])=\inf_{0\neq{u}\in{H^\gamma(M)}}E(u).
\end{equation}

A feature in all the proofs of the convergence of the Yamabe flow is the use at some point the so-called Positive Mass Theorem, as has already been present in \cite{schoen,brendle1,brendle2}. This is associated to the Green's function. Suppose $M$ is the conformal infinity of a Poincar\'{e}--Einstein manifold $(X^{n+1},g_+)$. Assume $Y_\gamma(M,[g])>0$ and $\lambda_1(g_+)\geq\frac{n^2}{2}-\gamma^2$. Then for each $y\in M$, there exists a Green's function $G(x,y)$ on $\bar{X}\backslash\{y\}$ (see \cite[Prosposition 1.5]{Kim2018}). In the fractional case, the Positive Mass Conjecture can be formulated in terms of the expansion of Green's function.
\begin{conj}
Assume that $\gamma\in (0,1)$, $n>2\gamma$ and $(M,[g])$ is a Poincar\'{e}--Einstein manifold with $Y_{\gamma}(M,[g])>0$. Fix any $y\in M$. Then there exists a small neighborhood of $y$ in $(\bar X,\bar g)$, which is diffeomorphic to a small neighborhood $\mathcal{N}\subset\mathbb{R}_+^{n+1}$ of 0, such that
\begin{equation*}
G(x,0)=g_{n,\gamma}|x|^{-(n-2\gamma)}+A+\psi(x)\quad \text{for }x\in \mathcal{N}
\end{equation*}
Here $g_{n,\gamma}=\pi^{-n/2}2^{-2\gamma}\Gamma(\gamma)^{-1}\Gamma(\frac{n}{2}-\gamma)$ and $\psi$ is a function in $\mathcal{N}$ satisfying
\[
|\psi(x)|\leq C|x|^{\min\{1,2\gamma\}}
\quad \text{ and } \quad
|\nabla \psi(x)|\leq C|x|^{\min\{0,2\gamma-1\}}
\]
for some constant $C>0$.
\end{conj}

The Positive Mass Theorem for the operators $P^g_\gamma$ even for $\gamma \in (0,1)$ is out of reach at the moment, for several reasons due to the non-locality assumption of the operator and the lack of tools to treat this case. So we naturally assume that Positive Mass holds in our main theorem as follows.

\begin{thm}\label{main}
For $\gamma\in(0,1)$, assume that $Y_\gamma(M,[g])>0$, $\lambda_1(g_+)\geq \frac{n^2}{2}-\gamma^2$ and,  in the case $\gamma\in (\frac{1}{2},1)$, $H=0$, where $H$ denotes the mean curvature of $\partial_\infty X=M$. Assume also the Positive Mass Conjecture holds with $A>0$. If $E$ is initially small in the sense that\footnote{Indeed, since $u(0)=1$, the initial energy is given by
\[E(u(0))=\frac{\int_{M}R(0)\,d\mu_0}{\mu_0(M)^{\frac{n-2\gamma}{n}}}=s_0.\]
}
\begin{equation}\label{eq:s0}
s_0\leq\left[(Y_\gamma(M,[g]))^{\frac{n}{2\gamma}}
        +Y_\gamma(\mathbb{S}^n)^{\frac{n}{2\gamma}}
    \right]^{\frac{2\gamma}{n}},
\end{equation}
then the flow \eqref{eq:flow} converges.
\end{thm}

\begin{remark} For $\gamma=\frac12$, \cite{almaraz5} has proved the convergence of flow under more general assumptions.
As previously mentioned, the operators $P^g_\gamma$, hence the fractional curvatures, are defined for every number (up to resonances) between $0$ and $n/2$. However, several major difficulties arise when one considers $\gamma >1$. First the maximum principle fails at the elliptic level and second the parabolic theory is completely open in this range. We leave as an open problem the investigation of these higher order curvatures. However, we will mention in the present paper the argument working in the larger range $\gamma >1$.
\end{remark}

\begin{remark}
In our main theorem, we didn't specify in which sense the flow converges. Following previous works, the flow is globally defined and H\"older continuous. It is an open question to prove that this is actually smooth, though such a result is expected. Implicitly, we assume the flow to be smooth in order to use Simon's inequality. The only proof of smoothness of the flow is in the Euclidean setting (see \cite{Vsmooth}) and the proof does not adapt straightforwardly to the manifold case. We postpone such result to future work.
\end{remark}

Let us also remark that, on the other hand, singular solutions do exist, at least for the elliptic problem. For the classical Yamabe problem, solutions with a prescribed singular set have been constructed by Mazzeo and Pacard \cite{Mazzeo-Pacard96} in 1996. This is recently extended by Ao, DelaTorre, Fontelos, Gonz\'{a}lez, Wei and the first author \cite{ACDFGW18} to the fractional case $\gamma\in(0,1)$. By a result of Gonz\'{a}lez, Mazzeo and the second author \cite{GMS}, the dimension $k$ of the singularity satisfies an inequality that includes in particular $k<(n-2\gamma)/2$. When $\gamma=1$, such dimension restriction is sharp according to the celebrated result of Schoen and Yau \cite{schoen-yau}. This is also known to Chang, Hang and Yang \cite{Chang-Hang-Yang04} when $\gamma=2$.

Our strategy follows the one in the book of M. Struwe \cite{bookStruwe} simplifying his original argument in virtue of the works of Brendle \cite{brendle1, brendle2}. This is based on a series of curvature bounds which allow compactness and a
recent global compactness result \cite{Palatucci2015}
in the spirit of Struwe's original one, developed
by Palatucci and Pisante (holding actually for all powers of $\gamma \in (0,n/2)$). The nonlocality of the flow induces several difficulties that one has to overcome using new inequalities which will be described over the paper.

\section{Preliminaries and technical tools}

In this section, we provide several tools to deal with our conformally covariant operators of fractional orders.

We will always assume that
$(M,g_0)$ is the conformal infinity of $(X,\bar{g}_0)$, both
equipped with appropriate metrics, and $\rho$ is
the associated defining function. Details can be
found in \cite{CG}.

\begin{prop}[Integration by parts]\label{lem:byparts}
Assume $H=0$ when $\gamma\in (\frac12,1)$. For any $v,w\in{C}^{\infty}(M)$, we have
\[\int_{M}P_{\gamma}^{g}(v)w\,d\mu=\int_{M}P_{\gamma}^{g}(w)v\,d\mu.\]
\end{prop}

\begin{proof}
We recall the ``improved'' extension \cite{CG} for $\Pgo-R_{g_0}$ without the zeroth order term, namely 
\begin{equation}\label{eq:extension}\begin{cases}
-\Div(\rho^{1-2\gamma}\nabla{W})
    =0
    &\textin(X,\bar{g}_0),\\
W=w
    &\texton(M,g_0),\\
(P_{\gamma}^{g_0}-R_{g_0})w=-c_\gamma\lim\limits_{\rho\to0}\rho^{1-2\gamma}\p_{\rho}W
    &\texton(M,g_0).
\end{cases}\end{equation}
where $c_\gamma$ is a positive constant (which can be found \cite{CG}).

First we prove that
\[\int_{M}P_{\gamma}^{g_0}(v)w\,d\mu_0=\int_{M}P_{\gamma}^{g_0}(w)v\,d\mu_0.\]
Indeed, denoting $V$ and $W$ to be the extension of $v$ and $w$ respectively, we have
\[\begin{split}
\int_{M}\left(P_{\gamma}^{g_0}(v)w-P_{\gamma}^{g_0}(w)v\right)\,d\mu_0
&=\int_{M}\left((P_{\gamma}^{g_0}-R)(v)w-(P_{\gamma}^{g_0}-R)(w)v\right)\,d\mu_0\\
&=c_{\gamma}\lim_{\rho\to0}\int_{M_\rho}\rho^{1-2\gamma}\left(W\p_{\rho}V-V\p_{\rho}W\right)\,d\bar{\mu}_0\\
&=c_{\gamma}\int_{X}\Div\left(\rho^{1-2\gamma}\left(W\nabla{V}-V\nabla{W}\right)\right)\,d\bar{\mu}_0\\
&=0.
\end{split}\]
Here $M_\rho$ denotes the level set at level $\rho$. For a conformal metric $g=u^{\frac{4}{n-2\gamma}}g_0$, we have
\[\int_{M}P_{\gamma}^{g}(v)w\,d\mu=\int_{M}u^{-\frac{n+2\gamma}{n-2\gamma}}P_{\gamma}^{g_0}(uv)wu^{\frac{2n}{n-2\gamma}}\,d\mu_0=\int_{M}P_{\gamma}^{g_0}(uv)uw\,d\mu_0.\]
Hence the result follows.
\end{proof}

We now compute crucial quantities involving the time-derivatives of $R$ and $s$. These computations can be justified by a standard approximation argument. Hereafter we also write $R(t)=R_\gamma^{g(t)}$, etc.

\begin{lem}\label{lem:Rt}
We have
\begin{enumerate}
\item $\p_tR(t)=\frac{n+2\gamma}{n-2\gamma}R(R-s)-P_{\gamma}^{g}(R-s)=-(P_{\gamma}^{g}-R)(R)+\frac{4\gamma}{n-2\gamma}R(R-s)$.
\item $\p_ts(t)=-2\int_{M}|R-s|^2\,d\mu$.
\end{enumerate}
\end{lem}

\begin{proof}
\begin{enumerate}
\item 
Using the definition of the flow, we have
\[\begin{split}
\p_{t}R(t)
&=-\dfrac{n+2\gamma}{n-2\gamma}\dfrac{u_t}{u}R
    +u^{-\frac{n+2\gamma}{n-2\gamma}}P_{\gamma}^{g_0}\left(\dfrac{u_t}{u}u\right)\\
&=\dfrac{n+2\gamma}{n-2\gamma}R(R-s)-P_{\gamma}^{g}(R-s).
\end{split}\]
\item 
Similarly we compute, using additionally Lemma \ref{lem:byparts},
\[\begin{split}
\p_ts(t)
&=\p_t\int_{M}R\,d\mu\\
&=\int_{M}R_t\,d\mu+\dfrac{2n}{n-2\gamma}\int_{M}R\dfrac{u_t}{u}\,d\mu\\
&=\dfrac{n+2\gamma}{n-2\gamma}\int_{M}R(R-s)\,d\mu
    -\int_{M}P_\gamma^{g}(R-s)\,d\mu\\
&\quad\;
    -\dfrac{2n}{n-2\gamma}\int_{M}R(R-s)\,d\mu\\
&=\left(\dfrac{n+2\gamma}{n-2\gamma}-1-\dfrac{2n}{n-2\gamma}\right)\int_{M}(R-s)^2\,d\mu\\
&=-2\int_{M}(R-s)^2\,d\mu.
\end{split}\]
\end{enumerate}
This completes the proof.
\end{proof}

Next we show that $R(t)\geq0$ for all $t$ provided that $R(0)>0$. Quantitatively we have

\begin{lem}
For any $t\geq0$, we have
\[R(t)\geq{e}^{-\frac{4\gamma}{n-2\gamma}s(0)t}\min_{M}{R}(0)>0.\]
\end{lem}
\begin{proof}
The extension problem for Lemma \ref{lem:Rt}(1) reads, for $U|_{M}=e^{\frac{4\gamma t}{n-2\gamma}}R(t)$,
\[\begin{cases}
\Div(\rho^{1-2\gamma}U)=0
    &\textin(X,\bar{g})\\
-c_{\gamma}\lim\limits_{\rho\to0}\rho^{1-2\gamma}\p_{\rho}U=-\p_{t}U+\dfrac{4\gamma}{n-2\gamma}U(R-s+s(0))
    &\textin(M,g_0).
\end{cases}\]
Testing this with $V$ such that $V|_{M}=\min\set{U-\min_{M}U(0),0}$, we have, as long as $R(t)>0$,
\[\begin{split}
0&=\int_{X}-\Div(\rho^{1-2\gamma}U)V\,d\bar{\mu}\\
&=\int_{X}\rho^{1-2\gamma}\nabla{U}\cdot\nabla{V}\,d\bar{\mu}
    -\int_{M}(P_{\gamma}^{g_0}-R)(U)V\,d\mu\\
&=\int_{X}\rho^{1-2\gamma}|\nabla V|^2\,d\bar{\mu}
    +\dfrac12\p_{t}\int_{M}V^2\,d\mu
    -\dfrac{4\gamma}{n-2\gamma}\int_{M}U(R-s+s(0))V\,d\mu\\
&\geq\int_{X}\rho^{1-2\gamma}|\nabla V|^2\,d\bar{\mu}
    +\dfrac12\p_{t}\int_{M}V^2\,d\mu
    +\dfrac{4\gamma}{n-2\gamma}\int_{M}U|V|R\,d\mu\\
&\geq\dfrac12\p_{t}\int_{M}V^2\,d\mu
\end{split}\]
where we have used the facts that $V\leq0$ and $s\leq{s(0)}$. Integrating in $t$, we obtain $V\equiv0$ up to the time where $\min_{M}U\geq0$, so that
\[R(t)>e^{-\frac{4\gamma}{n-2\gamma}s(0)t}\min_{M}R(0),\]
as desired.
\end{proof}

\begin{prop}\label{prop:Brendle-2.4}
Given any $T>0$, we can find positive constants $C(T)$ such that
$$C(T)^{-1}\leq u(t)\leq C(T)$$
for all $0\leq t\leq T$.
\end{prop}
\begin{proof}
The function $u(t)$ satisfies
\begin{align}
\partial_t u=-(R-s)u\leq s(0)u
\end{align}
then $u(t)\leq e^{s(0)T}$ for $0\leq t\leq T$. Since $R(t)\geq 0$ for $0\leq t\leq T$, then
\begin{align}
P_{\gamma}^{g_0}u=R(t)u^{\frac{n+2\gamma}{n-2\gamma}}\geq 0.
\end{align}
It follows from \cite[Lemma 4.9]{Cabre2014} that $u$ satisfies a Harnack inequality such that
\[\inf_M u\geq C(T)\sup_M u,\]
for some $C(T)>0$. Then the proposition is proved.
\end{proof}

For $q\geq1$ consider the functionals
\begin{equation}\label{eq:Fdef}
S_q(g)=\int_{M}(R_\gamma^{g})^q\,d\mu_g,\quad F_q(g)=\int_{M}|R_\gamma^{g}-s_\gamma^{g}|^{q}\,d\mu_g.
\end{equation}
In particular $s_\gamma^g=S_1(g)$.

\begin{lem}\label{lem:Brendle-lem2.5}
For $1\leq{q}<\frac{n}{2\gamma}$, we have
\begin{equation}\label{eq:liminfFq}
F_{q+1}(g(t))\leq C(T,q,g_0),
\end{equation}
for all $0\leq t\leq T$. If the flow exists for all $t>0$, then
\begin{equation}
\liminf_{t\to\infty}F_{q+1}(g(t))=0.
\end{equation}
\end{lem}

\begin{proof}
We compute, for $q\in[1,\frac{n}{2\gamma})$,
\begin{equation}\label{eq:dt-Sq}\begin{split}
&\quad\,\p_tS_q(g)\\
&=\int_{M}qR^{q-1}R_t\,d\mu+\dfrac{2n}{n-2\gamma}\int_{M}R^{q}\dfrac{u_t}{u}\,d\mu\\
&=-q\int_{M}(P_{\gamma}^{g_0}-R)(R-s)R^{q-1}\,d\mu
    +q\dfrac{4\gamma}{n-2\gamma}\int_{M}R(R-s)R^{q-1}\,d\mu
    \\
    &\qquad -\dfrac{2n}{n-2\gamma}\int_{M}R^q(R-s)\,d\mu\\
&=-q\int_{M}(P_{\gamma}^{g_0}-R)(R)R^{q-1}\,d\mu
    +\dfrac{2(2\gamma{q}-n)}{n-2\gamma}\int_{M}R^q(R-s)\,d\mu\\
&\leq-\dfrac{2(n-2\gamma{q})}{n-2\gamma}\int_{M}R^q(R-s)\,d\mu\leq0,\\
\end{split}\end{equation}
the last inequality following from the extension problem \eqref{eq:extension}. Indeed, if $U$ is the extension for $R$, then
\[\begin{split}
\int_{M}(P_{\gamma}^{g_0}-R)(R)R^{q-1}\,d\mu
&=-c_\gamma\lim\limits_{\rho\to0}\rho^{1-2\gamma}
    \int_{M}\p_{\rho}U R^{q-1}\,d\bar{\mu}\\
&=c_{\gamma}\int_{X}\rho^{1-2\gamma}
    \nabla{U}\cdot\nabla\left(U^{q-1}\right)\,d\bar{\mu}\\
&=\dfrac{4(q-1)}{q^2}c_{\gamma}
    \int_{X}\rho^{1-2\gamma}\left|\nabla(U^{\frac{q}{2}})\right|^2\,d\bar{\mu}\\
&\geq0.
\end{split}\]
Integrating \eqref{eq:dt-Sq}, we have
\begin{equation*}\begin{split}
\int_{0}^{\infty}F_{q+1}(g(t))\,dt
&=\int_0^{\infty}\int_{M}|R-s|^{q+1}\,d{\mu}dt\\
&\leq\int_0^{\infty}\int_{M}(R^q-s^q)(R-s)\,d{\mu}dt\\
&\leq\dfrac{n-2\gamma}{2(n-2\gamma{q})}S_q(g_0).
\end{split}\end{equation*}
In particular,
\[\liminf_{t\to\infty}F_{q+1}(g(t))=0.\]
\end{proof}

\section{Long time existence and convergence}
The short time and long time existence of $u$ has been studied by \cite{Daskalopoulos2017}. One can use the method in \cite{Athanasopoulos2010} to show that for any $T>0$, $u\in C^\alpha((0,T]\times M)$ for some $\alpha$. Here we are providing a proof follows from Brendle's approach \cite{brendle1}.

\begin{prop}
For any fixed $\frac{n}{2\gamma}<p<\frac{n+2\gamma}{2\gamma}$, let $\alpha=2\gamma-\frac{n}{p}>0$. Then for any $T>0$, there exists a constant $C(T)$ such that
\begin{align*}
|u(x_1,t_1)-u(x_2,t_2)|\leq C(T)((t_1-t_2)^{\frac{\alpha}{2}}+d(x_1,x_2)^\alpha).
\end{align*}
\end{prop}
\begin{proof}
Using Lemma \ref{lem:Brendle-lem2.5} and Proposition \ref{prop:Brendle-2.4}, for $\frac{n}{2\gamma}<p<\frac{n+2\gamma}{2\gamma}$,
\begin{align}\label{eq:W2p}
\int_M |P_\gamma^{g_0}u(t)|^pd\mu\leq C(T)
\end{align}
for all $0\leq t\leq T$ and
\begin{align}\label{eq:tup}
\int_M |\partial_t u|^pd\mu\leq C(T).
\end{align}
By \cite[Theorem 4]{Grubb2015}, the inequality \eqref{eq:W2p} implies that
\[|u(x,t)-u(y,t)|\leq C(T)d(x,y)^\alpha\]
where $\alpha=2\gamma-\frac{n}{p}$ and $t\in [0,T]$. Using \eqref{eq:tup}, we obtain
\begin{align*}
&|u(x,t_1)-u(x,t_2)|\\
\leq &C(t_1-t_2)^{-\frac{n}{2}}\int_{B_{\sqrt{t_1-t_2}}(x)}|u(x,t_1)-u(x,t_2)|\,d\mu_0(y)\\
\leq &C(t_1-t_2)^{-\frac{n}{2}}\int_{B_{\sqrt{t_1-t_2}}(x)}|u(y,t_1)-u(y,t_2)|\,d\mu_0(y)
    +C(T)(t_1-t_2)^{\frac{\alpha}{2}}\\
\leq &C(t_1-t_2)^{-\frac{n-2}{2}}\sup_{t\in [t_1,t_2]}\int_{B_{\sqrt{t_1-t_2}}(x)}|\partial_tu|\,d\mu_{0}(y)
    +C(T)(t_1-t_2)^{\frac{\alpha}{2}}\\
\leq &C(t_1-t_2)^{1-\frac{n}{2p}}\sup_{t\in [t_1,t_2]}\left(\int_{M}|\partial_tu|^p\,d\mu_{0}(y)\right)^{\frac{1}{p}}+C(T)(t_1-t_2)^{\frac{\alpha}{2}}\\
\leq& C(T)(t_1-t_2)^{\gamma-\frac{n}{2p}}
\end{align*}
for all $x\in M$ and $t_1,t_2\in [0,T]$ satisfying $0<t_1-t_2<1$. Thus the assertion is proved.
\end{proof}

Now we show that the convergence is uniform.

\begin{lem}\label{lem:Fpto0}
For any $p\in[1,\frac{n+2\gamma}{2\gamma})$, $F_p(g(t))\to0$ as $t\to+\infty$.
\end{lem}

\begin{proof}
We use the notation $z^p=|z|^{p-1}z$. Using the Stroock--Varopoulos inequality
\[\int_{M}f^{p-1}\Pg{f}\,d\mu
    \geq\dfrac{4(p-1)}{p^2}\int_{M}|f|^{\frac{p}{2}}\Pg\left(|f|^{\frac{p}{2}}\right)\,d\mu\]
together with the Sobolev inequality
\[
0<Y_\gamma(M,[g])=\inf_{0\neq{f}\in{C}^{\infty}(M)}
    \dfrac{\int_{M}f\Pg{f}\,d\mu}
        {\left(\int_{M}|f|^{\frac{2n}{n-2\gamma}}\,d\mu\right)^{\frac{n-2\gamma}{n}}},
\]
we compute
\[\begin{split}
\p_{t}F_p(g)
&=-p\int_{M}(R-s)^{p-1}\Pg(R-s)\,d\mu
    +p\dfrac{n+2\gamma}{n-2\gamma}\int_{M}|R-s|^{p}\,d\mu\\
&\quad\,+2ps_t\int_{M}(R-s)^{p-1}\,d\mu
    +\dfrac{2n}{n-2\gamma}\int_{M}|R-s|^{p}\dfrac{u_t}{u}\,d\mu\\
&\leq-\dfrac{4(p-1)}{pY_\gamma(M,[g])}F_{p^*}(g)^{\frac{n-2\gamma}{n}}
    +\dfrac{p(n+2\gamma)-2n}{n-2\gamma}F_{p+1}\\
&\quad\,
    +\dfrac{p(n+2\gamma)}{n-2\gamma}sF_p
    +2pF_2(g)F_{p-1}(g)
\end{split}\]
where we denote $p^*=\frac{np}{n-2\gamma}$. Using H\"{o}lder's inequality with the conjugate exponents $\theta=\frac{n-2\gamma}{2\gamma{p}}$ and $1-\theta=\frac{2\gamma(p+1)-n}{2\gamma{p}}$ such that $p+1=\theta{p^*}+(1-\theta)p$,
and Young's inequality with $\alpha=\frac{n}{2\gamma{p}}<1$, we have
\[\begin{split}
F_2(g)F_{p-1}(g)
    &\leq{F}_{p+1}(g)\\
F_{p+1}(g)
    &\leq{F}_{p^*}(g)^{\frac{n-2\gamma}{2\gamma{p}}}
        F_{p}(g)^{\frac{2\gamma(p+1)-n}{2\gamma{p}}}\\
    &\leq\delta{F}_{p^*}^{\frac{n-2\gamma}{n}}
        +C(\delta)F_{p}(g)^{1+\frac{2\gamma}{2\gamma{p}-n}}
\end{split}\]
for any $\delta>0$. Combining with the above estimates, we have
\[\begin{split}
\p_{t}F_{p}(g)
&\leq{C}F_{p}(g)+CF_{p}(g)^{1+\beta},
\end{split}\]
with $\beta=\frac{1}{p(1-\alpha)}\frac{2\gamma}{2\gamma{p}-n}>0$. Recalling \eqref{eq:liminfFq}, standard ODE analysis implies that
\[\lim_{t\to\infty}F_p(g(t))=0.\]

\end{proof}

Now we have proved $u(t)$ exists for $(0,\infty)$ and it is h\"{o}lder in space and time for any finite time interval. We want to study the convergence of $u(t)$.

Let $r_0>0$ denote a lower bound for the injectivity radius on $(M,g_0)$. Fix $\varphi\in{C}_c^{\infty}(B_{r_0}(0))$ such that $\varphi=1$ on $B_{r_0/2}\subset\R^n$. For $x,y\in{M}$, let $\varphi_y(x)=\varphi(\exp_y^{-1}(x))$, where $\exp$ is the exponential map in the metric $g_0$. Let us also denote, for functions $u$ and $\bar{u}$ defined on $M$ and $\R^n$ respectively,
\[\vol(u)=\mu(M)=\int_{M}u^{\frac{2n}{n-2\gamma}}\,d\mu_0,\qquad
    \volbar(\bar{u})=\int_{\R^n}\bar{u}^{\frac{2n}{n-2\gamma}}\,dx.\]

For any sequence of time, we have the profile decomposition by \cite{Palatucci2015}.

\begin{lem}
For any sequence $t_k\to \infty$, there exists an integer $L$ and sequences $x_{k,l}$, $\eps_{k,l}$, $l=1\dots L$ such that, passing to a subsequence if necessary,
\begin{equation}\label{eq:profileconv}
    u(t_k)-\sum_{l=1}^L{u}_{(x_{k,l},\eps_{k,l})}\to u_\infty
    \quad\text{ in }H^\gamma (M,g_0),
\end{equation}
where $u_\infty\geq 0$ solves
\begin{align}\label{eq:uinf-main}
    P^{g_0}_\gamma u_\infty=s_\infty u_\infty^{\frac{n+2\gamma}{n-2\gamma}}
        \quad\texton(M,g_0),
\end{align}
and
\[u_{(x_{k,l},\eps_{k,l})}(x)
    =\varphi_{x_{k,l}}(x)
        \bar{u}\left(\eps_{k,l}^{-1}\exp_{x_{k,l}}^{-1}(x)\right),
\]
with $\bar u=\bar{\alpha}_{n,\gamma}\left(1+|x|^2\right)^{-\frac{n-2\gamma}{2}}$, the standard bubble solving
\begin{align*}
    (-\Delta_{\mathbb{R}^n})^\gamma \bar u
    =\bar u^\frac{n+2\gamma}{n-2\gamma}
        \quad\texton\R^n.
\end{align*}
Moreover,
\begin{align}\label{eq:vol_decomp}
    \vol(u(t_k))=\vol(u_\infty)+L \cdot \volbar(\bar u).
\end{align}
\end{lem}

\begin{proof}
By \cite{Palatucci2015}, such profile decomposition holds as long as the Palais--Smale condition is verified\footnote{The authors proved the result in $\R^n$. In the manifold setting, the proof is almost identical.}. Indeed, from Lemma \ref{lem:Fpto0},
\begin{align*}
    \int_M |P_\gamma^{g_0}u-su^{\frac{n+2\gamma}{n-2\gamma}}|^{\frac{2n}{n+2\gamma}}d\mu_0=\int_M |R-s|^{\frac{2n}{n+2\gamma}}d\mu\to 0.
\end{align*}
Hence the result follows.
\end{proof}

Actually \eqref{eq:vol_decomp} means
\[1=\left(\frac{E(u_\infty)}{s_\infty}\right)^{\frac{n}{2\gamma}}+L\left(\frac{Y_\gamma(\mathbb{S}^n)}{s_\infty}\right)^{\frac{n}{2\gamma}},\]
which is
\[s_\infty=\left[E(u_\infty)^{\frac{n}{2\gamma}}
    +LY_\gamma(\mathbb{S}^n)^{\frac{n}{2\gamma}}\right]^{\frac{2\gamma}{n}}.\]
Obviously $E(u_\infty)\geq Y_\gamma(M,[g])$. By Lemma \ref{lem:Rt} and the assumption \eqref{eq:s0},
\[s_\infty\leq s_0\leq\left[(Y_\gamma(M,[g]))^{\frac{n}{2\gamma}}
    +Y_\gamma(\mathbb{S}^n)^{\frac{n}{2\gamma}}\right]^{\frac{2\gamma}{n}}.\]
By the Aubin inequality (see \cite{QG})
\begin{equation}\label{eq:Aubin-ineq}
Y_\gamma(M,[g])\leq Y_\gamma(\mathbb{S}^n),
\end{equation}
we conclude that either
\begin{enumerate}
    \item $L=0$ and $u_\infty>0$; or
    \item $L=1$ and $u_\infty\equiv 0$.
\end{enumerate}
\begin{remark}
The Aubin inequality \eqref{eq:Aubin-ineq} can be proved using concentration-compactness (as opposed to test functions) for any $\gamma \in (0,n/2)$.
\end{remark}

Using the version of strong maximum principle, again proved in \cite{QG}, one has
\begin{lem}
Either $u_\infty>0$ or $u_\infty\equiv 0$.
\end{lem}

Thus the above cases are a dichotomy and are to be referred to as the \emph{compact case} and the \emph{noncompact case} respectively.

The following proposition is crucial in proving the convergence, and its proof is the content of Sections \ref{sec:compact}--\ref{sec:noncompact}.

\begin{prop}\label{prop:4.16}
For any sequence $t_k\to \infty$ there exist constants $\delta\in (0,1)$ such that for a subsequence there holds
\[s(t_k)-s_\infty\leq CF_{\frac{2n}{n+2\gamma}}(g(t_k))^{\frac{n+2\gamma}{2n}(1+\delta)}.\]
\end{prop}

One consequence of this proposition is
\begin{lem}
There exist constants $\delta\in (0,1)$ and $T$ such that for all $t>T$ there holds
\[s(t)-s_\infty\leq CF_2(g(t))^{\frac{1+\delta}{2}}.\]
\end{lem}
\begin{proof}
Suppose this is not true. One can find a sequence $t_k\to \infty$ such that
\[s(t_k)-s_\infty\geq F_2(g(t_k))^{\frac{1+1/k}{2}}.\]
However, Proposition \ref{prop:4.16} can be applied to this sequence
\[s(t_k)-s_\infty\leq CF_{\frac{2n}{n+2\gamma}}(g(t_k))^{\frac{n+2\gamma}{2n}(1+\delta)}\leq CF_2(g(t_k))^{\frac{1+\delta}{2}},\]
the last inequality following from H\"{o}lder's inequality. Putting the two inequalities together, we obtain
\[1\leq CF_2(g(t_k))^{\frac{\delta-1/k}{2}},\]
which contradicts Lemma \ref{lem:Fpto0} when $k$ is sufficiently large.
\end{proof}

\begin{lem}
We have
\[\int_0^\infty F_2(g(t))^{\frac12}dt<\infty.\]
\end{lem}
\begin{proof}
Recall the relation
\[
    \frac{d}{dt}(s(t)-s_\infty)=-2F_2(g(t))\leq -C(s(t)-s_\infty)^{\frac{2}{1+\delta}}
\]
where $\delta \in (0,1)$. This differential inequality implies
\[s(t)-s_\infty\leq Ct^{-\frac{1+\delta}{1-\delta}}\]
for some constant $C>0$ and $t$ sufficiently large. Using H\"{o}lder's inequality, we obtain
\[\int_T^{2T}F_2(g(t))^{\frac{1}{2}}dt\leq T^\frac{1}{2}\left(\int_T^{2T}F_2(g(t))dt\right)^{\frac{1}{2}}\leq \frac{T^\frac{1}{2}}{4}(s(T)-s(2T))^\frac{1}{2}\leq CT^{-\frac{\delta}{1-\delta}}\]
if $T$ sufficently large. Since $\delta\in (0,1)$, we conclude that
\[\int_1^\infty F_2(g(t))^{\frac{1}{2}}dt\leq \sum_{k=0}^\infty\int_{2^k}^{2^{k+1}}F_2(g(t))^{\frac{1}{2}}dt\leq C\sum_{k=0}^{\infty}2^{-\frac{\delta}{1-\delta }k}\leq C.\]
\end{proof}
\begin{prop}\label{prop:3.6}
Given any $\epsilon_0>0$, there exists a real number $r>0$ and $q>\frac{n}{2
\gamma}$ such that
\[\int_{B_r(x)}|R(g(t))|^qd\mu(t)\leq \epsilon_0\]
for all $x\in M $ and $t\geq 0$.
\end{prop}
\begin{proof}
We can find a real number $T>0$ such that
\[\int_T^\infty\left(\int_M u(t)^{\frac{2n}{n-2\gamma}}(R(g(t)-s(t))^2d\mu_0\right)\leq \frac{\epsilon_0}{n}.\]
Choosing a real number $r>0$ such that
\[\int_{B_r(x)}u(t)^{\frac{2n}{n-2\gamma}}d\mu_{0}\leq \frac{\epsilon_0}{2}\]
for all $x\in M$ and $0\leq t\leq T$. Then for any $t\geq T$, we have
\[\begin{split}
&\int_{B_r(x)}u(t)^{\frac{2n}{n-2\gamma}}d\mu_0\\
\leq&\int_{B_r(x)}u(T)^{\frac{2n}{n-2\gamma}}d\mu_0+\frac n2 \int_T^\infty\left(\int_M u(t)^{\frac{2n}{n-2\gamma}}(R(g(t)-s(t))^2d\mu_0\right)\\
\leq&\epsilon_0.
\end{split}\]

From Lemma \ref{lem:Fpto0}, we can find $p,q\in(\frac{n}{2\gamma},\frac{n}{2\gamma}+1)$ such that $q<p$ and
\[\int_M |R(g(t))|^pd\mu_{g(t)}\leq C\]
for some constant $C$ independent of $t$. By the previous part of this proof, one can find $r>0$ independent of $t$ such that
\[\int_{B_r(x)}d\mu(t)\leq \epsilon_0.\]
Using H\"{o}lder's inequality, then
\[\int_{B_r(x)}|R(g(t))|^qd\mu(t)\leq C\epsilon_0^{1-\frac{q}{p}}.\]
Since $\epsilon_0$ can be chosen arbitrarily small, the proposition is proved.
\end{proof}

\begin{prop}
Suppose the assumption of Thm \ref{main} are satisfied, then we have the uniform upper and lower bound of $u(t)$, that is
\[\sup_Mu(t)\leq C,\quad \inf_M u(t)\geq C^{-1}\]
here $C$ is a positive constant independent of $t$.
\end{prop}
\begin{proof}
We will need Proposition \ref{prop:uni_upper} and verify its assumption is satisfied. Since our flow is volume preserving, then
\[\int_M d\mu_{g(t)}=1\]
and Lemma \ref{lem:Rt} implies that
\[\int_M R(g(t))d\mu_{g(t)}=\int_M u(t)\Pgo u(t)d\mu_0\leq s_0.\]
Now Proposition \ref{prop:3.6} means that we can find a uniform radius for any point $x\in M$ and $t>0$. Therefore we can arrive at an uniform upper bound of $u$ by Proposition \ref{prop:uni_upper}. For the lower bound of $u$, it is just a consequence of the Harnack inequality of  \cite{Cabre2014}.

\end{proof}



Our next goal is to prove Proposition \ref{prop:4.16} for the two cases.
\section{The compact case}\label{sec:compact}
In this case we have $u_\infty>0$. We first need a spectral decomposition with respect to weighted eigenfunctions of $\Pgo$.
\begin{prop}
There exist sequences $\{\psi_a\}_{a\in\mathbb{N}}\subset C^{\infty}(M)$ and $\{\lambda_a\}_{a\in\mathbb{N}}\subset \R$, with $\lambda_a>0$, satisfying:

(i) For all $a\in\N$,
\begin{equation}\notag
\Pgo\psi_a=\lambda_a u_\infty^{\frac{4\gamma}{n-2\gamma}}\psi_a\,,\textin M.
\end{equation}

(ii) For all $a,b\in\N$,
\[
\int_{M}\psi_a\psi_bu_{\infty}^{\frac{4\gamma}{n-2\gamma}}d\mu_0=
\begin{cases}
1\,,&\text{if}\:a=b\,,
\\
0\,,&\text{if}\:a\neq b\,.
\end{cases}
\]

(iii) The span of $\{\psi_a\}_{a\in\N}$ is dense in $L^2(M)$.

(iv) We have $\lim_{a\to\infty}\lambda_a=\infty$.
\end{prop}

\begin{proof}
Since we are assuming $R_{g_0}>0$, for each $f\in L^2(M)$ we can define $T(f)=u$, where $u\in H^{\gamma}(M)$ is the unique solution of
\[
\Pgo u=f u_{\infty}^\frac{4\gamma}{n-2\gamma}\textin\:M\,,
\]
It has been proved in \cite{QG} that the first eigenvalue of the operator is positive, hence
\[\int_M u P^{g_0}_\gamma u \,d\mu_0\]
defines an (equivalent) norm in $H^\gamma(M)$ which is compactly embedded in $L^2(M)$, and the operator $T:L^2( M)\to L^2(M)$ is compact. Integrating by parts, we see that $T$ is symmetric with respect to the inner product
\begin{equation}\label{eq:ip}
(\psi_1,\psi_2)
\mapsto
\int_{M}\psi_1\psi_2u_{\infty}^{\frac{4\gamma}{n-2\gamma}}\,d\mu_0.
\end{equation}
Then the result follows from the spectral theorem for compact operators.
\end{proof}

\begin{cor}\label{cor:CS}
For any $u,v\in{H^\gamma}(M)$, we have
\[\int_{M}u\Pgo{v}\,d\mu_0
    \leq\norm[H^\gamma]{u}\norm[H^\gamma]{v}.\]
\end{cor}

\begin{proof}
This follows directly from the eigenfunction expansion. If $u=\sum_{a}\mu_a\psi_a$ and $v=\sum_{b}\nu_b\psi_b$, then
\[\begin{split}
\int_{M}u\Pgo{v}\,d\mu_0
&=\int_{M}\sum_{a,b}\mu_a\psi_a
    \cdot\nu_b\lambda_{b}u_\infty^{\frac{4\gamma}{n-2\gamma}}\psi_b\,d\mu_0\\
&=\sum_{a,b}\lambda_{b}\mu_a\nu_b\delta_{ab}\\
&\leq\sqrt{\sum_{a}\lambda_a\mu_a^2}\sqrt{\sum_{a}\lambda_a\nu_a^2}\\
&=\norm[H^\gamma]{u}\norm[H^\gamma]{v}.
\end{split}\]
\end{proof}

Our next goal is show the coercivity in $H^\gamma(M,g_0)$ of the second variation operator of the Yamabe functional
at certain error $w_k$ (defined below in \eqref{eq:w_k}). This is the content of Proposition \ref{Propo6.9} and requires a projection onto a finite dimensional subspace that we now introduce.

Let $A\subset \N$ be a finite set such that $\lambda_a>\frac{n+2\gamma}{n-2\gamma}s_\infty$ for all $a\notin A$, and define the projection
\[
\Pi(f)
=\sum_{a\notin A}\left(\int_{ M}\psi_af d\mu_0\right)\psi_au_\infty^{\frac{4\gamma}{n-2\gamma}}
=f-\sum_{a\in A}\left(\int_{ M}\psi_af d\mu_0\right)\psi_a\uinf^{\frac{4\gamma}{n-2\gamma}}\,.
\]

Note that this definition facilitates the computations for the lemma below and is not the canonical projection with respect to the inner product defined in \eqref{eq:ip}, which would read
\[\tilde{\Pi}(f)
=\sum_{a\notin{A}}
    \left(
        \int_{M}\psi_{a}fu_{\infty}^{\frac{4\gamma}{n-2\gamma}}\,d\mu_0
    \right)\psi_a.
\]
We are going to construct functions $\bar{u}_z$, which are perturbations of $\uinf$ in a finite dimensional subspace, and whose derivatives satisfy nice orthogonality conditions. 

\begin{lem}\label{Lemma6.4}
There exists $\zeta>0$ with the following significance: for all $z=(z_1,...,z_{|A|})\in\R^{|A|}$ with $|z|\leq \zeta$, there exists a smooth function $\bar{u}_z$ satisfying,
\begin{equation}\label{Lemma6.4:1}
\int_{M}\uinf^{\frac{4\gamma}{n-2\gamma}}(\bar{u}_z-\uinf)\psi_a \,d\mu_0=z_a
\:\:\:\:\text{for all}\quad\:a\in A\,,
\end{equation}
and
\begin{equation}\label{Lemma6.4:2}
\Pi\left(\Pgo\bar{u}_z-s_{\infty}\bar{u}_{z}^{\frac{n+2\gamma}{n-2\gamma}}\right)=0\,.
\end{equation}
Moreover, the mapping $z\mapsto \bar{u}_z$ is real analytic.

As a result,
\end{lem}

\begin{proof}
This is just an application of the implicit function theorem and a standard argument to reach real analyticity.
\end{proof}
\begin{lem}\label{Lemma6.5}
There exists $0<\delta<1$ such that
\[
E(\bar{u}_z)-E(\uinf)
\leq
C\sup_{a\in A}\left|
\int_{ M}\psi_a\left(\Pgo\bar{u}_z-s_\infty\,\bar{u}_z^{\frac{n+2\gamma}{n-2\gamma}}
\right)d\mu_0
\right|^{1+\delta}\,,
\]
if $|z|$ is sufficiently small.
\end{lem}
\begin{proof}
Observe that the function $z\mapsto E(\bar{u}_z)$ is real analytic. According to results of {\L}ojasiewicz (see equation (2.4) in \cite[p.538]{Simon1983}), there exists $0<\delta<1$ such that
\[
|E(\bar{u}_z)-E(\uinf)|\leq \sup_{a\in A}\left|\frac{\partial}{\partial z_a}E(\bar{u}_z)\right|^{1+\delta}\,,
\]
if $|z|$ is sufficiently small. Now we can follow the lines in \cite[Lemma 6.5]{brendle1} to calculate the partial derivative of the function $z\mapsto E(\bar u_z)$,
\[
\begin{split}
    \frac{\partial}{\partial z_a}E(\bar{u}_z)
    &=2\dfrac{\displaystyle\int_{M}
        \left(
            \Pgo\bar{u}_z-s_\infty\bar{u}_z^{\frac{n+2\gamma}{n-2\gamma}}
        \right)\tilde{\psi}_{a,z}\,d\mu_0}
        {\left(\displaystyle\int_{M}
            \bar{u}_z^{\frac{2n}{n-2\gamma}}\,d\mu_0
        \right)^{\frac{n-2\gamma}{n}}}\\
    &\quad\;-2\left(
        \dfrac{\displaystyle\int_{M}
            \bar u_z\Pgo\bar u_z \,d\mu_0}
        {\displaystyle\int_{M}\bar{u}_z^{\frac{2n}{n-2\gamma}}
            \,d\mu_0}-s_\infty
        \right)
        \dfrac{\displaystyle\int_M \bar u_z^{\frac{n+2\gamma}{n-2\gamma}}\tilde{\psi}_{a,z}\,d\mu_0}
        {\left(\displaystyle\int_M \bar u_z^{\frac{2n}{n-2\gamma}}\,d\mu_0\right)^{\frac{n-2\gamma}{n}}},
\end{split}
\]
where $\tilde{\psi}_{a,z}:=\frac{\partial}{\partial z_a}\bar u_z$ for $a\in A$. According to \eqref{Lemma6.4:1} and \eqref{Lemma6.4:2}, we know that $\tilde{\psi}_{a,z}$ satisfies
\[\int_M \uinf^{\frac{4\gamma}{n-2\gamma}}\tilde{\psi}_{a,z}\psi_b\,d\mu_0
=\begin{cases}
1\quad& a=b,\\
0\quad& a\neq b,
\end{cases}\]
for $b\in A$ and $\Pi\left(\Pgo \tilde{\psi}_{a,z}-s_\infty \bar{u}_{z}^{\frac{4\gamma}{n-2\gamma}}\tilde{\psi}_{a,z}\right)=0$. Moreover, \eqref{Lemma6.4:2} implies
\[\Pgo\bar{u}_z-s_{\infty}\bar{u}_{z}^{\frac{n+2\gamma}{n-2\gamma}}
    =\sum_{b\in{A}}
        \left(
            \int_{M}\left(\Pgo\bar{u}_z-s_{\infty}\bar{u}_{z}^{\frac{n+2\gamma}{n-2\gamma}}\right)
                \psi_{b}u_{\infty}^{\frac{4\gamma}{n-2\gamma}}\,d\mu_0
        \right)\psi_b.\]
We therefore obtain
\[\begin{split}
    \frac{\partial}{\partial z_a}E(\bar u_z)
    &=2\dfrac{\displaystyle\int_{M}
        \left(\Pgo \bar{u}_z
            -s_{\infty}\bar{u}_z^{\frac{n+2\gamma}{n-2\gamma}}
        \right)\psi_a\,d\mu_0}
        {\left(\displaystyle\int_M \bar u_z^{\frac{2n}{n-2\gamma}}\,d\mu_0\right)^{\frac{n-2\gamma}{n}}}\\
    &\quad\;+2\sum_{b\in A}\dfrac{
            \displaystyle\int_{M}\left(\Pgo\bar u_z -s_\infty\bar u_z^{\frac{n+2\gamma}{n-2\gamma}} \right)\psi_b\,d\mu_0
            \int_{M}
            u_\infty^{\frac{4\gamma}{n-2\gamma}}\bar u_z\psi_{b}\,d\mu_0}
        {\displaystyle\int_{M}\bar u_z^{\frac{2n}{n-2\gamma}}\,d\mu_0}\\
    &\quad\qquad\qquad\cdot
        \dfrac{\displaystyle\int_M \bar u_z^{\frac{n+2\gamma}{n-2\gamma}}\tilde{\psi}_{a,z}\,d\mu_0}
        {\left(\displaystyle\int_{M}\bar{u}_z^{\frac{2n}{n-2\gamma}}
            \,d\mu_0\right)^{\frac{n-2\gamma}{n}}},
\end{split}\]
for all $a\in A$. Then the bounds for $u_{\infty}$ and $\bar{u}_z$ yield
\[\sup_{a\in A}\left|\frac{\partial}{\partial z_a}E(\bar u_z)\right|\leq C\sup_{a\in A}\left|\int_M\psi_a\left(\Pgo\bar u_z-s_\infty\bar u_z^{\frac{n+2\gamma}{n-2\gamma}}\right)d\mu_0\right|.\]
From this, the lemma follows.
\end{proof}

For any $k\in\mathbb{N}$, we consider the best approximation in $H^\gamma(M)$ of $u_k=u(t_k)$ among the family $\set{\bar{u}_z}$.
More precisely, we choose $z_k$ with $|z_k|\leq\zeta$ such that
\begin{align*}
    \int_M(\bar u_{z_k}-u_k)\Pgo(\bar u_{z_k}-u_k)\,d\mu_0
    =\min_{|z|\leq\zeta}
        \int_M (\bar u_z-u_k)\Pgo(\bar u_z-u_k)\,d\mu_0
\end{align*}
By \eqref{eq:profileconv}, we have $u_{k}\to u_\infty$ in $H^\gamma(M)$. As $\bar{u}_{0}=\uinf$, this implies that $z_k\to 0$ as $k\to \infty$. One can decompose
\begin{equation}\label{eq:w_k}
u_k=\bar u_{z_k}+w_k,
\end{equation}
such that
\[\norm[H^\gamma]{w_k}
    \to0\quad\textas{k\to\infty}.\]

It also follows from the variational properties of $\bar u_{z_k}$ that
\[\int_M \Pgo(w_k)\tilde{\psi}_{a,z_k}\,d\mu_0=0,\]
for any $a\in A$.
Again noticing the fact that $z_k\to0$ as $k\to\infty$ one can deduce, via Corollary \ref{cor:CS}, that
\begin{equation}\begin{split}\label{eq:orth_wk}
\lambda_a\int_M\uinf^{\frac{4\gamma}{n-2\gamma}}\psi_a\,w_k\,d\mu_0
&=\int_{M}w_k\Pgo\psi_a\,d\mu_0\\
&=\int_{M}(\psi_a-\tilde{\psi}_{a,z_k})\Pgo{w_k}\,d\mu_0\\
&\leq\norm[H^\gamma]{\psi_a-\tilde{\psi}_{a,z_k}}
    \norm[H^\gamma]{w_k}\\
&=o(1)\|w_k\|_{H^\gamma}
\end{split}\end{equation}
for any $a\in A$.
\begin{prop}\label{Propo6.9}
There exists $c>0$ such that
\begin{align*}
\frac{n+2\gamma}{n-2\gamma}s_\infty\int_{ M}\uinf^{\frac{4\gamma}{n-2\gamma}}w_k^2\,d\mu_0\leq
(1-c)\int_{M}w_k\Pgo w_k\,d\mu_0,
\end{align*}
for all $k$ sufficiently large.
\end{prop}

\begin{proof}
Suppose this were not true. Then there would be a subsequence, still denoted $w_k$, such that we may rescale them to $\tilde w_k$ satisfying
\[1=\int_{M}\tilde w_k\Pgo \tilde w_k\,d\mu_0\leq \liminf_{k\to \infty}\frac{n+2\gamma}{n-2\gamma}s_\infty\int_{ M}\uinf^{\frac{4\gamma}{n-2\gamma}}w_k^2\,d\mu_0.\]
Then $\tilde w_k$ is bounded in $H^\gamma$ and consequently $\tilde w_k\rightharpoonup \tilde w$ weakly in $H^\gamma$ for some $\tilde{w}$. The above inequality implies in particular that
\[1\leq\frac{n+2\gamma}{n-2\gamma}s_\infty\int_{ M}\uinf^{\frac{4\gamma}{n-2\gamma}}\tilde{w}^2\,d\mu_0,\]
so that $\tilde{w}\not\equiv0$. On the other hand,
\[\int_M \tilde w\Pgo\tilde{w}\,d\mu_0
    \leq\frac{n+2\gamma}{n-2\gamma}s_\infty
        \int_{M}\uinf^{\frac{4\gamma}{n-2\gamma}}\tilde{w}^2
            \,d\mu_0,
\]
or
\[\sum_{a\in\mathbb{N}}\lambda_a
    \left(\int_{M}\uinf^{\frac{4\gamma}{n-2\gamma}}\psi_a\tilde{w}\,d\mu_0\right)^2
\leq\sum_{a\in\mathbb{N}}\dfrac{n+2\gamma}{n-2\gamma}s_{\infty}
    \left(\int_{M}\uinf^{\frac{4\gamma}{n-2\gamma}}\psi_a\tilde{w}\,d\mu_0\right)^2.\]
However, \eqref{eq:orth_wk} shows that
\[\lambda_a\int_{M}\uinf^{\frac{4\gamma}{n-2\gamma}}\psi_a\tilde{w}\,d\mu_0=0,\]
for any $a\in A$, from which we arrive at a contradiction by the choice of $A$.
\end{proof}

We now estimate $w_k$ quantitatively.
\begin{lem}\label{lem:4.22}
There exist constants $C>0$ and $k_0$ such that for $k\geq k_0$ there holds
\[
\|w_k\|_{H^\gamma}
\leq C\left(\int_M
        \left|
            R(t_k)-s_\infty
        \right|^{\frac{2n}{n+2\gamma}}\,d\mu_{g(t_k)}
    \right)^{\frac{n+2\gamma}{2n}}.
\]
\end{lem}

We need some elementary inequalities, see also \cite[(137),\,(156)]{brendle1}.
\begin{lem}\label{lem:pointwise}
Let $a,b>0$. For $p>0$, we have
\[|a^p-b^p|
    \leq{C}|a-b|^p+Ca^{p-1}|a-b|.\]
For $p>1$,
\[|a^p-b^p-pa^{p-1}(a-b)|
    \leq{C}a^{\max\set{p-2,0}}|a-b|^{\min\set{p,2}}
        +C|a-b|^p.\]
Moreover, for $p>2$,
\begin{multline*}
\left|a^p-b^p-pa^{p-1}(a-b)
+\dfrac{p(p-1)}{2}b^{p-2}(a-b)^2\right|\\
    \leq{C}a^{\max\set{p-3,0}}|a-b|^{\min\set{p,3}}
        +C|a-b|^p.
\end{multline*}
\end{lem}
\begin{proof}
Let $h=a-b$. The first one follows directly from
\[|a^p-(a-h)^p|\leq\begin{cases}
Ca^{p-1}|h| & \textfor |h|\leq\frac{a}{2},\\
C|h|^p & \textfor |h|\geq\frac{a}{2}.
\end{cases}\]

The second estimate is similar when $p\geq2$, where we expand to the second order,
\[|a^p-(a-h)^p-pa^{p-1}h|\leq\begin{cases}
Ca^{p-2}|h|^2 & \textfor |h|\leq\frac{a}{2},\\
C|h|^p & \textfor |h|\geq\frac{a}{2}.
\end{cases}\]
When $p<2$, in the regime $|h|\leq\frac{a}{2}$ we simply bound $a^{p-2}\leq{C}|h|^{p-2}$.

The same argument applied to the last estimate reads
\[\left|a^p-(a-h)^p-pa^{p-1}h+\dfrac{p(p-1)}{2}(a-h)^{p-2}h^2\right|
\leq\begin{cases}
Ca^{p-3}|h|^3 & \textfor |h|\leq\frac{a}{2},\\
C|h|^p & \textfor |h|\geq\frac{a}{2},
\end{cases}
\]
hence the result.
\end{proof}

\begin{proof}[Proof of Lemma \ref{lem:4.22}] Denote $\hat w_k$ to be the projection of $w_k$ onto the subspace $\{\psi_a \mid a\not\in A\}$,
\[\hat w_k=\sum_{a\not\in A}\left(\int_M \uinf^{\frac{4\gamma}{n-2\gamma}}\psi_{a}w_{k}\,d\mu_0\right)\psi_a,\]
so that 
\[\int_{M}\uinf^{\frac{4\gamma}{n-2\gamma}}\psi_a\hat{w}_k\,d\mu_0=0\]
for any $a\in A$. Moreover, \eqref{eq:orth_wk} shows that
\[\begin{split}
    \int_M (\hat w_k-w_k)\Pgo(\hat w_k-w_k)d\mu_0
    &=\sum_{a\in A}\lambda_a\left(\int_M \uinf^{\frac{4\gamma}{n-2\gamma}}\psi_aw_kd\mu_0\right)^2
    \\
    &=o(1)\|w_k\|_{H^\gamma}^2.
\end{split}\]
In other words, $\|\hat w_k-w_k\|_{H^\gamma}=o(1)\|w_k\|_{H^\gamma}$. With the decomposition $u_k=\bar u_{z_k}+w_k$, we calculate the linearization
\begin{equation}\label{eq:expansion-R}
    \begin{split}
    &\quad\;\left(R(t_k)-s_\infty\right)u_k^{\frac{n+2\gamma}{n-2\gamma}}\\
    &=\Pgo{u_k}-s_{\infty}u_k^{\frac{n+2\gamma}{n-2\gamma}}\\
    &=\Pgo\bar u_{z_k}-s_\infty \bar u_{z_k}^{\frac{n+2\gamma}{n-2\gamma}}+\Pgo w_k-s_\infty u_k^{\frac{n+2\gamma}{n-2\gamma}}+s_\infty \bar u_{z_k}^{\frac{n+2\gamma}{n-2\gamma}}\\
    &=\Pgo\bar u_{z_k}-s_\infty \bar u_{z_k}^{\frac{n+2\gamma}{n-2\gamma}}+\Pgo w_k-\frac{n+2\gamma}{n-2\gamma}s_\infty u_\infty^{\frac{4\gamma}{n-2\gamma}}w_k+I_k,
    \end{split}
\end{equation}
where
\[\begin{split}
I_k
&=s_\infty
    \left(\bar{u}_{z_k}^{\frac{n+2\gamma}{n-2\gamma}}
        -u_k^{\frac{n+2\gamma}{n-2\gamma}}
        +\frac{n+2\gamma}{n-2\gamma}u_\infty^{\frac{4\gamma}{n-2\gamma}}w_k
    \right)\\
&=\frac{n+2\gamma}{n-2\gamma}s_\infty
    \left(\bar{u}_{z_k}^{\frac{4\gamma}{n-2\gamma}}-u_\infty^{\frac{4\gamma}{n-2\gamma}}
    \right)w_k\\
&\quad\;+s_\infty
    \left(\bar{u}_{z_k}^{\frac{n+2\gamma}{n-2\gamma}}
        -u_k^{\frac{n+2\gamma}{n-2\gamma}}
        +\frac{n+2\gamma}{n-2\gamma}\bar{u}_{z_k}^{\frac{4\gamma}{n-2\gamma}}w_k
    \right).
\end{split}\]
Using Lemma \ref{lem:pointwise},
\begin{equation*}\begin{split}
|I_k|
&\leq{C}\uinf^{\frac{6\gamma-n}{n-2\gamma}}
    |\bar{u}_{z_k}-\uinf||w_k|
    +C|\bar{u}_{z_k}-u_\infty|^{\frac{4\gamma}{n-2\gamma}}|w_k|\\
&\quad\;+C\bar{u}_{z_k}^{\max\set{0,\frac{4\gamma}{n-2\gamma}-1}}
        |w_k|^{\min\set{\frac{n+2\gamma}{n-2\gamma},2}}
    +C|w_k|^{\frac{n+2\gamma}{n-2\gamma}}
\end{split}\end{equation*}
By the Sobolev embedding $H^\gamma\hookrightarrow L^{\frac{2n}{n-2\gamma}}$ and the smallness of $\bar{u}_{z_k}-\uinf$ and $w_k$, we conclude
\begin{equation*}
\int_M |I_k\hat w_k|\,d\mu_0
\leq \int_M |I_k| |w_k|\,d\mu_0
\leq o(1)\|w_k\|_{H^\gamma}^2
\end{equation*}
as $k\to\infty$. (Note that we need $\uinf\geq{c}>0$ in case $6\gamma<n$.) Notice that the projection $\hat{w_k}$ satisfies
\begin{equation*}
    \int_M \left(\Pgo\bar u_{z_k}-s_\infty \bar{u}_{z_k}^{\frac{n+2\gamma}{n-2\gamma}}\right)
    \hat w_k\,d\mu_0=0
\end{equation*}
because of \eqref{Lemma6.4:2}. Now, using Proposition \ref{Propo6.9},
\begin{equation*}\begin{split}
    c\|w_k\|^2_{H^\gamma}
    &\leq \int_M \left(\Pgo w_k-\frac{n+2\gamma}{n-2\gamma}s_\infty \uinf^{\frac{4\gamma}{n-2\gamma}}w_k\right)w_k\,d\mu_0\\
    &=\int_{M}\left(\Pgo{w}_k
            -\frac{n+2\gamma}{n-2\gamma}s_\infty \uinf^{\frac{4\gamma}{n-2\gamma}}w_k
        \right)\hat w_k\,d\mu_0
        +o(1)\|w_k\|^2_{H^\gamma}\\
    &=\int_M \left(\left(R(t_k)-s_\infty\right)
            u_k^{\frac{n+2\gamma}{n-2\gamma}}
            \hat{w}_k
        -I_k\hat w_k\right)\,d\mu_0
        +o(1)\|w_k\|^2_{H^\gamma}\\
    &\leq{C}
        \left(\int_{M}
            \left|
                R(t_k)-s_\infty
            \right|^{\frac{2n}{n+2\gamma}}
            \,d\mu_{g(t_k)}
        \right)^{\frac{n+2\gamma}{2n}}
            \|w_k\|_{H^\gamma}
        +o(1)\|w_k\|^2_{H^\gamma}
\end{split}\end{equation*}
and the claim follows by making $k$ large enough.
\end{proof}

A related computation completes the estimate in Lemma \ref{Lemma6.5}.
\begin{lem}\label{lem:4.31}
There exist $C>0$ and $k_0$ such that for $k\geq k_0$ there holds
\[E(\bar{u}_{z_k})-E(u_\infty)
\leq{C}
    \left(\int_{M}
        \left|
            R(t_k)-s_\infty
        \right|^{\frac{2n}{n+2\gamma}}
        \,d\mu_{g(t_k)}
    \right)^{\frac{n+2\gamma}{2n}(1+\delta)}.\]
\end{lem}
\begin{proof}
Recalling \eqref{eq:expansion-R},
\[\Pgo \bar u_{z_k}-s_\infty \bar u_{z_k}^{\frac{n+2\gamma}{n-2\gamma}}
=\left(R(t_k)-s_\infty\right)
    u_k^{\frac{n+2\gamma}{n-2\gamma}}
-\Pgo w_k
-s_\infty\left(\bar{u}_{z_k}^\frac{n+2\gamma}{n-2\gamma}
    -u_k^\frac{n+2\gamma}{n-2\gamma}\right),
\]

it suffices to bound the projection of each term on the right hand side onto the finite dimensional subspace spanned by $\psi_a$, where $a\in{A}$. We have 
\[\int_M \Pgo(w_k)\psi_a \,d\mu_0
\leq C\|w_k\|_{H^\gamma}\]
by Corollary \ref{cor:CS}. Also,

\[\begin{split}
\int_{M}s_\infty
    \left(\bar{u}_{z_k}^\frac{n+2\gamma}{n-2\gamma}
        -u_k^\frac{n+2\gamma}{n-2\gamma}
    \right)\psi_a\,d\mu_0
&\leq{C}\int_{M}\left(u_k+|w_k|\right)^{\frac{4\gamma}{n-2\gamma}}
    |w_k\psi_a|\,d\mu_0\\
&\leq C\|w_k\|_{H^\gamma}.
\end{split}\]
Using Lemma \ref{lem:4.22}, we get
\begin{multline*}
\sup_{a\in A}\left|\int_M\left(\Pgo\bar u_{z_k}-s_\infty\bar u_{z_k}^{\frac{n+2\gamma}{n-2\gamma}}\right)
    \psi_a\,d\mu_0\right|\\
\leq{C}\left(\int_{M}
        \left|
            R(t_k)-s_\infty
        \right|^{\frac{2n}{n+2\gamma}}
        \,d\mu_{g(t_k)}
    \right)^{\frac{n+2\gamma}{2n}}.
\end{multline*}
Our claim follows from Lemma \ref{Lemma6.5}.
\end{proof}

We can finally turn to the proof of Proposition \ref{prop:4.16} in the compact case.

\begin{proof}[Proof of Proposition \ref{prop:4.16}]
By the conformal relation \eqref{eq:conformalP}, we compute
\[\begin{split}
E(u_k)
&=\int_{M}(\bar u_{z_k}+w_k)\Pgo(\bar{u}_{z_k}+w_k)
    \,d\mu_0\\
&=\int_M \bar{u}_{z_k}\Pgo\bar{u}_{z_k}\,d\mu_0
    +2\int_{M}R(t_k)u_k^{\frac{n+2\gamma}{n-2\gamma}}
        w_k\,d\mu_0
    -\int_M w_k\Pgo w_k\,d\mu_0\\
&=s_\infty
    +2\int_{M}\left(R(t_k)-s_\infty\right)
        u_k^{\frac{n+2\gamma}{n-2\gamma}}w_k\,d\mu_0\\
&\quad\;-\int_{M}
    \left(w_k\Pgo{w_k}
        -\frac{n+2\gamma}{n-2\gamma}s_\infty
        \bar{u}_{z_k}^{\frac{4\gamma}{n-2\gamma}}w_k^2
    \right)\,d\mu_0
    +J_k
\end{split}
\]
where
\[\begin{split}
    J_k&=(E(\bar u_{z_k})-s_\infty)\left(\int_M \bar u_{z_k}^{\frac{2n}{n-2\gamma}}d\mu_0\right)^{\frac{n-2\gamma}{n}}+s_\infty\left(\left(\int_M \bar u_{z_k}^{\frac{2n}{n-2\gamma}}d\mu_0\right)^{\frac{n-2\gamma}{n}}-1\right)\\
    &\quad\;+s_\infty\int_M \left(2 u_k^{\frac{n+2\gamma}{n-2\gamma}}w_k-\frac{n+2\gamma}{n-2\gamma}\bar u_{z_k}^{\frac{4\gamma}{n-2\gamma}}w_k^2\right)d\mu_0.
\end{split}\]
Since $x\mapsto x^{\frac{n-2\gamma}{n}}$ is a concave function,
\[\begin{split}
\left(\int_M \bar u_{z_k}^{\frac{2n}{n-2\gamma}}\,d\mu_0\right)^{\frac{n-2\gamma}{n}}-1
&\leq \frac{n-2\gamma}{n}\left(\int_M \bar u_{z_k}^{\frac{2n}{n-2\gamma}}\,d\mu_0-1\right)\\
&=\frac{n-2\gamma}{n}\int_M \left(\bar u_{z_k}^{\frac{2n}{n-2\gamma}}-u_{k}^{\frac{2n}{n-2\gamma}}\right)
    \,d\mu_0.
\end{split}\]
Then we can estimate the error term as
\[\begin{split}
    J_k
    &\leq (E(\bar u_{z_k})-s_\infty)\left(\int_M \bar u_{z_k}^{\frac{2n}{n-2\gamma}}d\mu_0\right)^{\frac{n-2\gamma}{n}}
    \\
    &\quad\;-s_\infty \int_M \Bigg(
        \frac{n-2\gamma}{n}u_k^{\frac{2n}{n-2\gamma}}
        -\frac{n-2\gamma}{n}\bar{u}_{z_k}^{\frac{2n}{n-2\gamma}}
    \\
    &\qquad\qquad\qquad
        -2u_k^{\frac{n+2\gamma}{n-2\gamma}}w_k
        +\frac{n+2\gamma}{n-2\gamma}\bar{u}_{z_k}^{\frac{4\gamma}{n-2\gamma}}w_k^2
    \Bigg)d\mu_0.
\end{split}
\]
Recalling $u_k=\bar{u}_{z_k}+w_k$, the last integrand is a multiple of
\[\left(\bar{u}_{z_k}+w_k\right)^{\frac{n-2\gamma}{n}}
    -\bar{u}_{z_k}^{\frac{2n}{n-2\gamma}}
    -\dfrac{2n}{n-2\gamma}
        \left(\bar{u}_{z_k}+w_k\right)^{\frac{2n}{n-2\gamma}}w_k
    +\dfrac12\cdot\dfrac{2n}{n-2\gamma}\cdot\dfrac{n+2\gamma}{n-2\gamma}
        \bar{u}_{z_k}^{\frac{4\gamma}{n-2\gamma}}w_k^2
\]
hence the pointwise estimate in Lemma \ref{lem:pointwise} applies to yield
\[\begin{split}
    &\quad\;\int_M\left|
        \frac{n-2\gamma}{n}u_k^{\frac{2n}{n-2\gamma}}
        -\frac{n-2\gamma}{n}\bar{u}_{z_k}^{\frac{2n}{n-2\gamma}}
        -2u_k^{\frac{n+2\gamma}{n-2\gamma}}w_k
        +\frac{n+2\gamma}{n-2\gamma}\bar{u}_{z_k}^{\frac{4\gamma}{n-2\gamma}}w_k^2
        \right|\,d\mu_0\\
    &\leq{C}\int_{M}
        \bar{u}_{z_k}^{\max\set{0,\frac{6\gamma-n}{n-2\gamma}}}
            |w_k|^{\min\set{\frac{2n}{n-2\gamma},3}}
        \,d\mu_0
        +C\int_{M}|w_k|^{\frac{2n}{n-2\gamma}} \,d\mu_0\\
    &\leq{C}\left(
            \int_{M}|w_k|^{\frac{2n}{n-2\gamma}}\,d\mu_0
        \right)^{\frac{n-2\gamma}{2n}\min\set{\frac{2n}{n-2\gamma},3}}\\
    &\leq{C}\|w_k\|_{H^\gamma}^{\min\set{\frac{2n}{n-2\gamma},3}}.
\end{split}
\]
Now the results of Lemma \ref{lem:4.22} and \ref{lem:4.31} imply
\[\begin{split}
    J_k
    &\leq (E(\bar{u}_{z_k})-s_\infty)
        \left(\int_{M}\bar{u}_{z_k}^{\frac{2n}{n-2\gamma}}
            \,d\mu_0\right)^{\frac{n-2\gamma}{n}}
        +C\|w_k\|_{H^\gamma}^{\min\set{\frac{2n}{n-2\gamma},3}}\\
    &\leq{C}
        \left(\int_M
            \left|
                R(t_k)-s_\infty
            \right|^{\frac{2n}{n+2\gamma}}\,d\mu_{g(t_k)}
        \right)^{\frac{n+2\gamma}{2n}(1+\delta)}\\
    &\leq{C}F_{\frac{2n}{n+2\gamma}}(g(t_k))^{\frac{n+2\gamma}{2n}(1+\delta)}
        +C(s(t_k)-s_\infty)^{1+\delta}.
\end{split}
\]

It follows from H\"{o}lder's inequality, Proposition \ref{Propo6.9} and Lemma \ref{lem:4.22} that the remaining terms are also bounded by
\[\begin{split}
    &\quad\;2\int_{M}\left(R(t_k)-s_\infty\right)
        u_k^{\frac{n+2\gamma}{n-2\gamma}}w_k\,d\mu_0\\
    &\qquad\qquad-\int_{M}
        \left(w_k\Pgo{w_k}
            -\frac{n+2\gamma}{n-2\gamma}s_{\infty}
                \bar{u}_{z_k}^{\frac{4\gamma}{n-2\gamma}}
                w_k^2
        \right)\,d\mu_0\\
    &\leq{C}
        \left(\int_{M}
            \left|
                R(t_k)-s_\infty
            \right|^{\frac{2n}{n+2\gamma}}
            \,d\mu_{g(t_k)}
        \right)^{\frac{n+2\gamma}{2n}}
            \|w_k\|_{L^{\frac{2n}{n-2\gamma}}}
        -c\|w_k\|^2_{H^\gamma}\\
    &\leq{C}\left(\int_M|R(t_k)-s_\infty|^{\frac{2n}{n+2\gamma}}d\mu_{g(t_k)}\right)^{\frac{n+2\gamma}{n}}\\
    &\leq{C}F_{\frac{2n}{n+2\gamma}}(g(t_k))^{\frac{n+2\gamma}{2n}(1+\delta)}
        +C(s(t_k)-s_\infty)^{1+\delta}.
\end{split}
\]

Combining our expansion of $E(u_k)$ and the previous estimates, we get
\[s(t_k)-s_\infty=E(u_k)-s_\infty\leq C F_{\frac{2n}{n+2\gamma}}(g(t_k))^{\frac{n+2\gamma}{2n}(1+\delta)}+C(s(t_k)-s_\infty)^{1+\delta}.\]
Since $s(t_k)\to s_\infty$ as $k\to \infty$ and $\delta\in (0,1)$, then
\[s(t_k)-s_\infty\leq CF_{\frac{2n}{n+2\gamma}}(g(t_k))^{\frac{n+2\gamma}{2n}(1+\delta)},\]
as desired.
\end{proof}

\section{The noncompact case}\label{sec:noncompact}
In this case we have $u_\infty=0$. Following \cite{Kim2018} with the assumption of Positive Mass Theorem for our operators, there is a test function $u$ such that
\[E(u)
=\frac{\displaystyle\int_{M}u\Pgo{u}\,d\mu_0}
    {\left(\displaystyle\int_{M}u^{\frac{2n}{n-2\gamma}}d\mu_0
        \right)^{\frac{n-2\gamma}{n}}}
<Y_{\gamma}(\mathbb{S}^n).\]
Such $u$ is found through the rescaling and relocation of standard bubble $\bar{u}$, possibly truncated or perturbed. By specifying the relocation and rescaling parameters $(x_0,\eps)$ of such test function, we use the notation $u_{(x_0,\eps)}$ for a more precise purpose. Near $x_0$, $u_{(x_0,\eps)}(x)$ is comparable to \[\bar\alpha_{n,\gamma}s_\infty^{-\frac{n-2\gamma}{4\gamma}}
    \eps^{-\frac{n-2\gamma}{2}}\bar{u}
    \left(\eps^{-1}\exp_{x_0}^{-1}(x)\right),
\]
where $\bar\alpha_{n,\gamma}$ can be found at \cite[1-23]{Kim2018}.

From the profile decomposition, we know that $u_k=u(t_k)$ approaches some $u_{(x_k,\eps_k)}$ in $H^\gamma$. We prefer to use the best approximation in the following sense,
\begin{multline*}
\int_M\left(u_k-\alpha_ku_{(x_k,\eps_k)}\right)
    \Pgo(u_k-\alpha_ku_{(x_k,\eps_k)})\,d\mu_0\\
=\min_{\alpha>0,x\in M,\eps>0}
    \int_M\left(u_k-\alpha u_{(x,\eps)}\right)
    \Pgo\left(u_k-\alpha u_{(x,\eps)}\right)\,d\mu_0.
\end{multline*}
Then
\[u_k=\alpha_ku_{(x_k,\eps_k)}+w_k=:v_k+w_k\]
with some suitable $x_k,\eps_k$ and $\alpha_k\to\const>0$.
Then we have the following lemma from the variation of three parameters $\alpha,\eps$, and $x$ respectively:

\begin{lem}\label{lem:5.1}
As $k\to \infty$, there hold
\begin{enumerate}
    \item $\displaystyle\int_M v_{k}^{\frac{n+2\gamma}{n-2\gamma}}w_k\,d\mu_0
            =o(1)\|w_k\|_{H^\gamma}$;
    \item $\displaystyle\int_M v_{k}^{\frac{n+2\gamma}{n-2\gamma}}\frac{\eps_k^2-d(x,x_k)^2}{\eps_k^2+d(x,x_k)^2}w_k\,d\mu_0
            =o(1)\|w_k\|_{H^\gamma}$;
    \item $\displaystyle\int_M v_{k}^{\frac{n+2\gamma}{n-2\gamma}}\frac{\eps_k \exp_{x_k}^{-1}(x)}{\eps_k^2+d(x,x_k)^2}w_k\,d\mu_0
            =o(1)\|w_k\|_{H^\gamma}$.
\end{enumerate}
\end{lem}

\begin{proof}
By the choice of $\alpha_k$, we get
\[\int_M w_k\Pgo u_{(x_k,\eps_k)}\,d\mu_0=0.\]
Moreover, one can expand
\[\alpha_k\Pgo{u}_{(x_k,\eps_k)}=\Pgo(u_k-w_k)=R(t_k)u_k^{\frac{n+2\gamma}{n-2\gamma}}-\Pgo w_k=s_\infty v_{k}^{\frac{n+2\gamma}{n-2\gamma}}+I_k-\Pgo w_k\]
where
\[I_k=\left(R(t_k)-s_\infty\right)
        u_k^{\frac{n+2\gamma}{n-2\gamma}}
    +s_\infty\left(u_k^{\frac{n+2\gamma}{n-2\gamma}}
        -v_{k}^{\frac{n+2\gamma}{n-2\gamma}}\right)
    \to 0 \quad\textin L^{\frac{2n}{n+2\gamma}}.\]
Then
\[s_\infty\int_M w_k\Pgo v_{k}^{\frac{n+2\gamma}{n-2\gamma}} \,d\mu_0
=o(1)\|w_k\|_{L^{\frac{2n}{n-2\gamma}}}+\|w_k\|_{H^\gamma}^2=o(1)\|w_k\|_{H^\gamma},\]
establishing Claim (1). Claim (2) and Claim (3) can be proved similarly.
\end{proof}

\begin{lem}\label{lem:orth-decomp-2}
There exist constants $c>0$ and $k_0$ such that for $k\geq k_0$ there holds
\[\frac{n+2\gamma}{n-2\gamma}s_\infty\int_M v_k^{\frac{4\gamma}{n-2\gamma}}w_k^2\,d\mu_0
\leq (1-c)\int_M w_k\Pgo w_k\,d\mu_0.\]
\end{lem}
\begin{proof}
Suppose it were not true. Then one would be able to extract a sequence of rescaled $\tilde w_k=a_kw_k$ such that
\[1=\int_M \tilde w_k\Pgo \tilde w_k\,d\mu_0\leq \liminf_{k\to \infty}\frac{n+2\gamma}{n-2\gamma}s_\infty\int_M v_k^{\frac{4\gamma}{n-2\gamma}}\tilde w_k^2\,d\mu_0\]
Define
\[\hat w_k(x)=\eps_k^{\frac{n-2\gamma}{2}}\tilde w_k(\exp_{x_k}(\eps_k\xi)):B_{R/\epsilon_k}(0)\subset T_{x_k}M\to \mathbb{R}\]
for some $R<\imath_0$, the injectivity radius of $(M,g_0)$. Then $\hat w_k$ is bounded in $H^\gamma(B_{R/\eps_k}(0))$ and consequently $\hat w_k\rightharpoonup \hat w$ weakly in $H^\gamma_{\loc}(\mathbb{R}^n)$ for some $\hat w$ satisfying
\[\int_{\mathbb{R}^n}\frac{\hat w(\xi)^2}{(1+|\xi|^2)^{2\gamma}}\,d\xi>0\]
and
\begin{align}\label{eq:second-var}
    \int_{\mathbb{R}^n}\hat w(\xi)(-\Delta_{\mathbb{R}^n})^\gamma \hat w(\xi)\,d\xi
    \leq \alpha_{n,\gamma}^{\frac{4\gamma}{n-2\gamma}}\frac{n+2\gamma}{n-2\gamma}\int_{\mathbb{R}^n}\frac{\hat w(\xi)^2}{(1+|\xi|^2)^{2\gamma}}\,d\xi.
\end{align}
However, it follows from Lemma \ref{lem:5.1} that
\begin{equation}\label{eq:3equalities}
\begin{split}
    \int_{\mathbb{R}^n}\left(\frac{1}{1+|\xi|^2}\right)^{\frac{n+2\gamma}{2}}\hat w(\xi)\,d\xi&=0,\\
    \int_{\mathbb{R}^n}\left(\frac{1}{1+|\xi|^2}\right)^{\frac{n+2\gamma}{2}}\frac{1-|\xi|^2}{1+|\xi|^2}\hat w(\xi)\,d\xi&=0,\\
    \int_{\mathbb{R}^n}\left(\frac{1}{1+|\xi|^2}\right)^{\frac{n+2\gamma}{2}}\frac{\xi}{1+|\xi|^2}\hat w(\xi)\,d\xi&=0.
\end{split}
\end{equation}
We want to prove the above three equalities and \eqref{eq:second-var} together imply $\hat w(\xi)=0$, which will clearly give us a contradiction. To this end, it is better to work on sphere $\mathbb{S}^n$. Denote by $\Sigma$ the stereographic projection of the sphere $\mathbb{S}^n$ onto $\mathbb{R}^n$ with respect to the north pole. More precisely,
\begin{equation*}
\begin{split}
    \forall\, x=&\,(x_1,\dots, x_{n+1})\in \mathbb{S}^n, \quad \Sigma(x)=\xi=(\xi_1,\dots, \xi_{n})\in \mathbb{R}^n\\
  &\text{where } \xi_i=\frac{x_i}{1-x_{n+1}}.
\end{split}
\end{equation*}
It is known that the standard metric of $\mathbb{S}^n$ and $\mathbb{R}^n$ are related by
\[g_{\mathbb{S}^n}=\frac{4}{(1+|\xi|^2)^2} |d\xi|^2=\rho(\xi)^{\frac{4}{n-2\gamma}}|d\xi|^2,\quad \rho(\xi)=\left(\frac{2}{1+|\xi|^2}\right)^{\frac{n-2\gamma}{2}}.\]
For any $\hat w(\xi)\in H^\gamma(\mathbb{R}^n)$, we define a function $v$ on $\mathbb{S}^n$ by $v(x)=(\rho^{-1}\hat w)(\xi)$, $\xi=\Sigma(x)$. The conformal property reads
\begin{align*}
    (-\Delta_{\mathbb{R}^n})^\gamma \hat w &=\rho^{\frac{n+2\gamma}{n-2\gamma}} P^{g_{\mathbb{S}^n}}_\gamma (v).
\end{align*}
Consequently,
\begin{align}
    \int_{\mathbb{R}^n}\hat w(-\Delta_{\mathbb{R}^n})^\gamma \hat w\,d\xi &=\int_{\mathbb{S}^n} v P^{g_{\mathbb{S}^n}}_\gamma (v) \,d\mu_{\mathbb{S}^n},\label{eq:wv-1}\\
    \int_{\mathbb{R}^n}\rho^{\frac{4\gamma}{n-2\gamma}}\hat w^2(\xi)\,d\xi
    &=\int_{\mathbb{S}^n} v^2(x)\,d\mu_{g_{\mathbb{S}^n}}.\label{eq:wv-2}
\end{align}
The spectrum of $P^{g_{\mathbb{S}^n}}_\gamma$ is known; for example, see \cite{JX}. Namely, for any $k\geq 0$
\begin{equation*}
    P^{g_{\mathbb{S}^n}}_\gamma(Y^{(k)})=\frac{\Gamma(k+\frac{n}{2}+\gamma)}{\Gamma(k+\frac{n}{2}-\gamma)}Y^{(k)},
\end{equation*}
where $Y^{(k)}$ are spherical harmonics of degree $k\geq 0$ and $\Gamma$ is the Gamma function. The three equalities in \eqref{eq:3equalities} mean exactly that $v$ is orthogonal to any $Y^{(0)}$ and $Y^{(1)}$. Therefore
\begin{align*}
\int_{\mathbb{S}^n} v P^{g_{\mathbb{S}^n}}_\gamma(v)\,d\mu_{\mathbb{S}^n}
\geq \frac{\Gamma(2+\frac{n}{2}+\gamma)}{\Gamma(2+\frac{n}{2}-\gamma)}\int_{\mathbb{S}^n} v^2(x)\,d\mu_{g_{\mathbb{S}^n}}.
\end{align*}
Combining the above fact with \eqref{eq:second-var}, \eqref{eq:wv-1} and \eqref{eq:wv-2}, we shall obtain
\begin{align*}
    \frac{\Gamma(2+\frac{n}{2}+\gamma)}{\Gamma(2+\frac{n}{2}-\gamma)}\int_{\mathbb{R}^n}\rho^{\frac{4\gamma}{n-2\gamma}}\hat w^2(\xi)\,d\xi
    &\leq \alpha_{n,\gamma}^{\frac{4\gamma}{n-2\gamma}}\frac{n+2\gamma}{n-2\gamma}\int_{\mathbb{R}^n}\frac{\hat w^2(\xi)}{(1+|\xi|^2)^{2\gamma}}\,d\xi\\
    &=\alpha_{n,\gamma}^{\frac{4\gamma}{n-2\gamma}}\frac{n+2\gamma}{n-2\gamma}2^{-2\gamma}\int_{\mathbb{R}^n}\rho(x)^{\frac{4\gamma}{n-2\gamma}}\hat w^2(\xi)\,d\xi.
\end{align*}
Retrieving $\alpha_{n,\gamma}$ from \cite{Kim2018} gives
\[\alpha_{n,\gamma}=2^{\frac{n-2\gamma}{2}}\left(\frac{\Gamma(\frac{n}{2}+\gamma)}{\Gamma(\frac{n}{2}-\gamma)}\right)^{\frac{n-2\gamma}{4\gamma}}.\]
It is not difficult to see
\[\frac{\Gamma(2+\frac{n}{2}+\gamma)}{\Gamma(2+\frac{n}{2}-\gamma)}>\alpha_{n,\gamma}^{\frac{4\gamma}{n-2\gamma}}\frac{n+2\gamma}{n-2\gamma}2^{-2\gamma}\]
for $\gamma\in (0,1)$ and $n>2\gamma$. Thus we conclude
\[\int_{\mathbb{R}^n}\rho^{\frac{4\gamma}{n-2\gamma}}\hat w^2(\xi)\,d\xi=0,\]
implying that $\hat{w}(\xi)\equiv0$, a contradiction.
\end{proof}

With the above estimate we now give the proof of Proposition \ref{prop:4.16} in the noncompact case.

\begin{proof}[Proof of Proposition \ref{prop:4.16}]
\[\begin{split}
    E(u_k)
    &=\int_M (v_k+w_k)\Pgo(v_k+w_k)\,d\mu_0\\
    &=\int_M v_k\Pgo v_k\,d\mu_0
        +2\int_{M}R(t_k)
            u_k^{\frac{n+2\gamma}{n-2\gamma}}
            w_k\,d\mu_0
        -\int_M w_k\Pgo w_k\,d\mu_0\\
    &=s_\infty+2\int_{M}
        \left(R(t_k)-s_\infty\right)
        u_k^{\frac{n+2\gamma}{n-2\gamma}}w_k\,d\mu_0\\
    &\quad\;-\int_{M}
        \left(w_k\Pgo{w_k}
            -\frac{n+2\gamma}{n-2\gamma}s_\infty v_k^{\frac{4\gamma}{n-2\gamma}}w_k^2
        \right)\,d\mu_0
        +J_k,
\end{split}
\]
where
\[\begin{split}
    J_k&=(E(v_k)-s_\infty)\left(\int_M v_k^{\frac{2n}{n-2\gamma}}d\mu_0\right)^{\frac{n-2\gamma}{n}}+s_\infty\left(\left(\int_M v_k^{\frac{2n}{n-2\gamma}}d\mu_0\right)^{\frac{n-2\gamma}{n}}-1\right)\\
    &\quad\;+s_\infty\int_M \left(2 u_k^{\frac{n+2\gamma}{n-2\gamma}}w_k-\frac{n+2\gamma}{n-2\gamma} v_k^{\frac{4\gamma}{n-2\gamma}}w_k^2\right)d\mu_0.
\end{split}\]
Since $E(v_k)<s_\infty$, we have
\begin{multline*}
J_k
\leq s_\infty \int_M \Bigg(
    -\frac{n-2\gamma}{n}u_k^{\frac{2n}{n-2\gamma}}
    +\frac{n-2\gamma}{n}v_k^{\frac{2n}{n-2\gamma}}
\\
    +2u_k^{\frac{n+2\gamma}{n-2\gamma}}w_k
    -\frac{n+2\gamma}{n-2\gamma}v_k^{\frac{4\gamma}{n-2\gamma}}w_k^2
    \Bigg)\,d\mu_0.
\end{multline*}
Similar to the case when $u_\infty>0$, one can get the estimate
\[J_k\leq C\|w_k\|_{H^\gamma}^{\min\set{\frac{2n}{n-2\gamma},3}}.\]
Lemma \ref{lem:orth-decomp-2} yields that
\[\|w_k\|^2_{H^\gamma}\leq C\int_M\left(w_k\Pgo{w_k}-\frac{n+2\gamma}{n-2\gamma}s_\infty v_k^{\frac{4}{n-2\gamma}}w_k^2\right)\,d\mu_0.\]
for $k\geq k_0$. The rest of proof follows from almost same lines as the compact case $u_\infty>0$.
\end{proof}

\appendix
\section{Some elliptic estimates}
Here we prove a Moser Harnack inequality; similar results can be found at \cite[Appendix A]{Almaraz2016} and \cite[Theorem 3.4]{QG}. For a fixed boundary point $(p_0,0)\in\partial{X}$, we consider local coordinates $(x,\rho)\in\R^n\times\R$ and use the notation
\[B_r^+=\set{(x,\rho)\in\bar {X}:\rho>0,\,d_{\bar g}((x,\rho),p_0)<r},\]
\[\Gamma_r^0=\set{(x,0)\in M:d_{g_0}(x,p_0)<r},\]
\[\Gamma_r^+=\set{(x,\rho)\in\bar X:\rho\geq0,\,d_{\bar g}((x,\rho),p_0)=r}.\]

\begin{prop}\label{prop:harnack}
Let $U$ be a nonnegative weak solution to
\[\left\{\begin{array}{@{\;}r@{\;}ll}
\Div(\rho^{1-2\gamma}\nabla{U})+E(\rho)U&=0
    \hfill\quad\textin{B}_{2r}^{+},\\
-\lim\limits_{\rho\to0}\rho^{1-2\gamma}\p_{\rho}U&=f(x)
    \hfill\quad\texton{\Gamma}_{2r}^{0}.
\end{array}\right.\]
where $|E(\rho)|\leq C\rho^{1-2\gamma}$. Then for each $\bar{p}>1$ and $q>\frac{n}{2\gamma}$,
\begin{multline*}
\sup_{B_r^+}U+\sup_{\Gamma_r^0}U
\\
\leq{C}_{\bar{p},q}
    \left[
        r^{-\frac{n+2-2\gamma}{\bar{p}}}\norm[L^{\bar{p}}(B_{2r}^+,\rho^{1-2\gamma})]{U}
        +r^{-\frac{n}{\bar{p}}}\norm[L^{\bar{p}}(\Gamma_{2r}^{0})]{U}+r^{2\gamma-\frac{n}{q}}||f||_{L^q(\Gamma_{2r}^0)}
    \right]
\end{multline*}
for some $C_{\bar{p},q}>0$ depending on $\bar{p}$ and $q$.

\end{prop}

\begin{proof}
The Moser iteration process is by now a very standard approach. We will just sketch the main steps. Details can be found in \cite{Almaraz2016} and \cite{QG}. Since we are just using the local information, we will prove the Harnack inequality in the Euclidean case and use $y>0$ as the extension variable.

After scaling we can assume $r=1$. Let $\ell=\|f\|_{L^q(\Gamma_{2}^0)}$ and $0\leq \eta\in C^1_c(B_2^+)$. We will work with the case $\ell>0$, for otherwise we may let an arbitrary positive $\ell$ tend to zero. Set $\bar U=U+\ell$ and, for simplicity, $a=1-2\gamma$. Firstly by multiplying the equation by $\eta^2 \bar U^\beta$ for some $\beta>0$ and integrating by parts, we have
\[\begin{split}
&\quad\;2\int_{B_2^+}y^a\eta \bar U^\beta \nabla \eta \nabla \bar U\,dxdy
    +\beta \int_{B_2^+}y^a\eta^2\bar U^{\beta -1}|\nabla \bar U|^2 \,dxdy
    +\int_{\Gamma_{2}^0}\eta^2\bar U^\beta f(x) \,dx\\
&=\int_{B^+_2}E(y)\eta^2 \bar U^{\beta+1}\,dxdy.
\end{split}\]
Using H\"{o}lder's inequality to handle the cross term, we simplify it using Young's inequality as
\begin{equation*}\begin{split}
\int_{B_2^+}y^a\eta^2\bar U^{\beta-1}|\nabla \bar U|^2\,dxdy
&\leq \frac{C}{\beta^2}\int_{B_2^+}y^a|\nabla \eta|\bar U^{\beta+1}\,dxdy
    +\frac{C}{\beta}\int_{\Gamma_2^0}\eta^2\frac{|f|}{\ell}\bar U^{\beta+1}\,dx\\
&\quad\;+\frac{C}{\beta}\int_{B_2^+}y^a\eta^2\bar U^{\beta+1}\,dx.
\end{split}\end{equation*}
Define $w=\bar U^{\frac{1+\beta}{2}}$ and insert it to the above equation. One gets
\begin{equation}
\begin{split}\label{eq:A-4}
&\quad\;\int_{B_2^+}y^a|\nabla (\eta w)|^2\,dxdy\\
&\leq C\frac{(\beta+1)^2}{\beta^2}
    \int_{B_2^+}y^a(|\nabla\eta|^2+\eta^2)w^2\,dxdy
    +C\frac{(\beta+1)^2}{\beta}\int_{\Gamma_{2}^0}\eta^2w^2 \frac{|f|}{\ell}\,dx\\
&=:I_1+I_2.
\end{split}
\end{equation}
For the left hand side above, one uses the trace Sobolev and weighted Sobolev embedding (see \cite[Corollary 5.3 and Proposition 3.3]{QG}) to obtain
\begin{equation}\label{eq:A-6}
C\int_{B_2^+}y^a|\nabla (\eta w)|^2\,dxdy
\geq \left(\int_{\Gamma_2^0}(\eta w)^{\frac{2n}{n-2\gamma}}\,dx\right)^{\frac{n-2\gamma}{n}}
    +\left(\int_{B_2^+}y^a(\eta w)^k \,dxdy\right)^{\frac{2}{k}},
\end{equation}
where $C>0$ is some constant and $k\in(1, 2(n+1)/n)$.

Next we estimate $I_2$ in \eqref{eq:A-4}. We have
\begin{equation}\label{eq:A-5}\begin{split}
\int_{\Gamma_2^0}\eta^2w^2\frac{|f|}{k}dx
&\leq \left\|\frac{|f|}{k}\right\|_{L^q(\Gamma_2^0)}\|\eta w\|_{L^{2q/(q-1)}(\Gamma_2^0)}^2\\
&\leq \epsilon\|\eta w\|_{L^{2n/(n-2\gamma)}(\Gamma_2^0)}^2
    +\epsilon^{-\frac{n}{2\gamma q-n}}\|\eta w\|_{L^2(\Gamma_2^0)}^2.
\end{split}\end{equation}
Choosing $\epsilon $ small enough, the first term of the right hand side can be absorbed  in to left hand side of \eqref{eq:A-6}. Plugging \eqref{eq:A-5} and \eqref{eq:A-6} back into \eqref{eq:A-4}, one gets
\begin{equation}
\begin{split}\label{eq:A-10}
&\quad\;\left(\int_{\Gamma_2^0}(\eta{w})^{\frac{2n}{n-2\gamma}}\,dx\right)^{\frac{n-2\gamma}{n}}
    +\left(\int_{B_2^+}y^a(\eta{w})^k\,dxdy\right)^{\frac{2}{k}}\\
&\leq C(1+\beta)^{\frac{4\gamma q}{2\gamma q-n}}\left[\int_{B_2^+}y^a(|\nabla \eta|^2+\eta^2)w^2\,dxdy
    +\int_{\Gamma_2^0}(\eta w)^2\,dx\right].
\end{split}
\end{equation}
For any $1\leq r_1\leq r_2\leq 2$, we choose $\eta$ as a cut-off function satisfying $0\leq \eta\leq 1$, $\eta\leq 2/(r_2-r_1)$ and $\eta=1$ in $B_{r_1}^+$ and $\eta=0$ on $B_2^+\backslash B_{r_2}^+$. With this $\eta$ in \eqref{eq:A-10}, we obtain, in terms of $\bar{U}$,
\begin{equation}
\begin{split}\label{eq:A-11}
&\quad\;\left(\int_{\Gamma_{r_1}^0}\bar U^{\frac{(\beta+1)n}{n-2\gamma}}\,dx\right)^{\frac{n-2\gamma}{n}}
    +\left(\int_{B_{r_1}^+}y^a\bar U^{(\beta+1)k}\,dxdy\right)^{\frac1k}\\
&\leq C\frac{(1+\beta)^\frac{4\gamma q}{2\gamma q-n}}{(r_2-r_1)^2}\left(\int_{\Gamma_{r_2}^0}\bar U^{\beta+1}\,dx
    +\int_{B_{r_2}^+}y^a\bar U^{\beta+1}\,dxdy\right).
\end{split}
\end{equation}
If we set
$$\Phi(p,r)=\left(\int_{\Gamma_{r}^0}\bar U^{p}\,dx\right)^{\frac{1}{p}}+\left(\int_{B_{r}^+}y^a\bar U^{p}\,dxdy\right)^{\frac1p}$$
and $\theta=\min\{\frac{n}{n-2\gamma},k\}>1$, then \eqref{eq:A-11} becomes
\[
\Phi(\theta (\beta+1),r_1)\leq \left(\frac{C(1+\beta)^{\frac{2\gamma q}{2\gamma q-n}}}{r_2-r_1}\right)^{\frac{2}{\beta+1}}\Phi(\beta+1, r_2).
\]
Now we can iterate the above inequality by setting $R_m=1+1/2^m$ and $\theta_m=\theta^m\bar{p}$. Then
\[\Phi(\theta_m,1)
    \leq \Phi(\theta_m,R_m)
    \leq(c_1\theta)^{c_2\sum_{i=0}^{m-1}i/\theta^i}\Phi(\bar{p},2)
    \leq C\Phi(\bar{p},2)
\]
for some constant $C$, because the series $\sum_{i=0}^\infty i/\theta^i$ is convergent. Finally, since
\[\lim_{p\to \infty}\Phi(p,1)=\sup_{\Gamma_1^0}\bar U+\sup_{B_1^+}\bar U,
\]
we have
\[\sup_{\Gamma_1^0}U+\sup_{B_1^+}U
\leq C\left[||U||_{L^{\bar{p}}(B_2^+,y^a)}
    +||U||_{L^{\bar{p}}(\Gamma_2^0)}
    +||f||_{L^{q}(\Gamma_2^0)}\right].
\]
Rescaling back to $B_{2r}^{+}$, we conclude the proof of theorem.
\end{proof}

\begin{prop}\label{prop:uni_upper}
Suppose $(M^n,g_0)$ is the comformal infinity of a Poincar\'{e}--Einstein manifold with $n>2\gamma$. For each $q>\frac{n}{2\gamma}$ we can find positive constants $\eta_0=\eta_0(M,g_0,q,C_1)$ and $C=C(M,g_0,q,C_1)$ with the following significance: if $g=u^\frac{4}{n-2\gamma}g_0$ is a conformal metric  and $R=P_{\gamma}^{g}(1)$ satisfying
\begin{equation}\label{eq:A3assump}
\int_M u^{\frac{2n}{n-2\gamma}}\,d\mu_0
    +\int_M u\Pgo u \,d\mu_0
    \leq C_1
    \quad\text{ and }\quad
    \int_{\Gamma_{2r}^{0}(x)}|R|^q\,d\mu_g\leq \eta_0
\end{equation}
for $x\in M$, then we have
\[u(x)
\leq C.
\]
\end{prop}

Before we prove this proposition, we collect some useful estimates.
\begin{lem}\label{lem:A-ineq}
Let $x\in{M}$. Under the same assumptions as in Proposition \ref{prop:uni_upper}, there hold
\begin{equation}\label{eq:A-est}
r^{-2}\int_{B_{2r}^+(x)}
        \rho^{1-2\gamma}U^2
    \,d\mu_{\bar{g}_0}
\leq{C_2}
\end{equation}
and
\[
r^{-n}
    \int_{\Gamma_{2r}^{0}(x)}d\mu_g
\leq{C_2},
\]
where $C_2$ depends only on $C_1$.
\end{lem}

\begin{proof}
For the first assertion, using H\"{o}lder's inequality,
\begin{equation*}
\int_{B_{2r}^+}
    \rho^{1-2\gamma}U^2
    \,d\mu_{\bar{g}_0}
\leq{C}r^2
    \left(
        \int_{B_{2r}^+}\rho^{1-2\gamma}U^{\frac{2(n+2-2\gamma)}{n-2\gamma}}
        \,d\mu_{\bar{g}_0}
    \right)^{\frac{n-2\gamma}{n+2-2\gamma}}.
\end{equation*}
It follows from the weighted Sobolev embedding in \cite[Theorem 2]{horiuchi1989} and the weighted Poincar\'{e}--Hardy inequalities (see \cite{liouville}) that
\begin{equation*}
\begin{split}
&\quad\;
    \left(\int_{B_{2r}^+}\rho^{1-2\gamma}U^{\frac{2(n+2-2\gamma)}{n-2\gamma}} \,d\mu_{\bar{g}_0} \right)^{\frac{n-2\gamma}{n+2-2\gamma}}\\
&\leq C\int_{B_{2r}^{+}}
    \rho^{1-2\gamma}|\nabla{U}|_{\bar{g}_0}^2\,d\mu_{\bar{g}_0}
    +C\int_{B_{2r}^{+}}
        \rho^{-1-2\gamma} U^2\,d\mu_{\bar g_0}\\
&\leq{C}\int_{X}
    \rho^{1-2\gamma}|\nabla{U}|_{\bar{g}_0}^2\,d\mu_{\bar{g}_0}\\
&\leq C\int_{M}u\Pgo u\,d\mu_{0}
    +C\int_{M}u^2\,d\mu_{0}
\leq CC_1,
\end{split}
\end{equation*}
where $C$ is large enough constant that depends only on $(X,\bar g_0)$.

The second estimate is immediate.
\end{proof}

\begin{proof}[Proof of Proposition \ref{prop:uni_upper}]
The proof is similar to \cite[Proposition A.3]{almaraz5} where the author deals with the $\gamma=\frac12$ case. The key step is to obtain \cite[(187)]{brendle1} in our setting. This is the consequence of Proposition \ref{prop:harnack}, which, we stress again, holds on the manifold $(M^n,g_0)$. Let $U$ be the extension of $u$ to $X$, which satisfies
\begin{equation}
\left\{\begin{array}{@{\;}r@{\;}ll}
-\Div(\rho^{1-2\gamma}\nabla{U})+E_{g_0}(\rho)U
    &=0
    &\textin(X,\bar{g}_0),\\
U&=u
    &\texton(M,g_0),\\
-c_\gamma\lim\limits_{\rho\to0}\rho^{1-2\gamma}\p_{\rho}U
    &=\Pgo u
    &\texton(M,g_0).
\end{array}\right.
\end{equation}
It follows from Proposition \ref{prop:harnack} that for any center on $M$ and any small radius $r>0$,
\begin{equation}
\begin{split}\label{eq:190.1}
\sup_{\Gamma_r^0}U+\sup_{B_r^+}U
&\leq
    Cr^{-\frac{n+2-2\gamma}{2}}\left(\int_{B_{2r}^+}\rho^{1-2\gamma} U^2
        \,d\mu_{\bar{g}_0}\right)^{\frac12}\\
&\quad\;
    +Cr^{-\frac{n-2\gamma}{2}}\left(\int_{\Gamma_{2r}^0}u^{\frac{2n}{n-2\gamma}}
        \,d\mu_0\right)^{\frac{n-2\gamma}{2n}}\\
&\quad\;
    +Cr^{2\gamma-\frac{n}{q}}\left(\int_{\Gamma_{2r}^0}|\Pgo{u}|^q
        \,d\mu_0\right)^{\frac1q}.
\end{split}
\end{equation}
Notice that Lemma \ref{lem:A-ineq} and our assumption \eqref{eq:A3assump} imply
\begin{equation}\label{eq:sup_har}\begin{split}
r^{\frac{n-2\gamma}{2}}\sup_{B_r^+(x)}U
&\leq C_2r^{\frac{n-2\gamma}{2}}
    +Cr^{\frac{n+2\gamma}{2}-\frac{n}{q}}
    \left(
        \int_{\Gamma_{2r}^0(x)}|\Pgo{u}|^q\,d\mu_0
    \right)^{\frac1q}.
\end{split}\end{equation}

Now let us suppose $r_0$ is a real number such that $r_0<r$ and
\[(r-s)^{\frac{n-2\gamma}{2}}\sup_{B^+_{s}(x)}U\leq (r-r_0)^{\frac{n-2\gamma}{2}}\sup_{B_{r_0}^+(x)}U\]
for all $s<r$. Moreover, we can find $x_0\in B_{r_0}^+(x)$ such that
\[\sup_{B_{r_0}^+(x)}U=U(x_0).\]
We can assume $r$ is small that $d(x_0,M)$ is achieved by a unique point $x_0^*$ on $M$.
By the definition of $r_0$ and $x_0$, we have
\[\sup_{B^+_{\frac{r-r_0}{2}}(x^*_0)}U\leq \sup_{B^+_{\frac{r+r_0}{2}}(x)}U\leq 2^{\frac{n-2\gamma}{2}}U(x_0).\]

We want to show that \eqref{eq:sup_har}  implies the existence of a fixed constant $K=K(C_2)$ such that for \emph{all} $s\leq \frac{r-r_0}{2}$,
\begin{equation}\label{eq:190.2}
s^{\frac{n-2\gamma}{2}}U(x_0)
\leq K
    +K(s^{\frac{n-2\gamma}{2}}U(x_0))^{\frac{n+2\gamma}{n-2\gamma}
    -\frac{2n}{n-2\gamma}\frac1q}\left(\int_{B_r^+(x)}|R_g|^q\,d\mu_g\right)^{\frac1q}.
\end{equation}
To that end we distinguish two cases according to the size of $s$. If $0<s<\min\{2d(x_0,M),\frac{r-r_0}{2}\}$, then the interior Harnack inequality yields
\[
    s^{\frac{n-2\gamma}{2}}U(x_0)
\leq
    Cs^{-1}\left(\int_{B_{s}}\rho^{1-2\gamma} U^2\,d\mu_{\bar{g}_0}\right)^{\frac12}
\leq K,
\]
for a ball $B_s\subset X$, using estimates similar to \eqref{eq:A-est}. On the other hand, if $\min\{2d(x_0,M),\frac{r-r_0}{2}\}\leq s\leq  \frac{r-r_0}{2}$, then $x_0\in B_s^+(x_0^*)$ and $\Gamma_{2s}^0(x_0^*)\subset \Gamma_{2r}^0(x)$. We get from \eqref{eq:sup_har} that
\begin{equation}\label{eq:190.3}\begin{split}
s^{\frac{n-2\gamma}{2}}U(x_0)
&\leq
    C_2s^{\frac{n-2\gamma}{2}}
    +C s^{\frac{n+2\gamma}{2}-\frac{n}{q}}
        \left(\int_{\Gamma_{2r}^0(x)}|\Pgo u|^q\,d\mu_0\right)^{1/q}\\
&\leq
    Ks^{\frac{n-2\gamma}{2}}
    +Ks^{\frac{n+2\gamma}{2}-\frac{n}{q}}
        \left(\int_{\Gamma_{2r}^0(x)}u^{\frac{n+2\gamma}{n-2\gamma}q-\frac{2n}{n-2\gamma}}|R|^q\,d\mu_g\right)^{1/q},
\end{split}\end{equation}
so again we get \eqref{eq:190.2}. Therefore \eqref{eq:190.2} holds for any $s\leq \frac{r-r_0}{2}$.

Now we choose $\eta_0>0$ such that
\[
(2K)^{\frac{n+2\gamma}{n-2\gamma}-\frac{2n}{n-2\gamma}\frac1q}\eta_0^{\frac1q}\leq \frac12.
\]

We claim that
\[\left(\frac{r-r_0}{2}\right)^{\frac{n-2\gamma}{2}}U(x_0)\leq 2K.\]
Indeed, if, on the contrary, $2K\leq (\frac{r-r_0}{2})^{\frac{n-2\gamma}{2}}U(x_0)$, then we let $s=(\frac{2K}{U(x_0)})^{\frac{2}{n-2\gamma}}\leq \frac{r-r_0}{2}$ in \eqref{eq:190.2}, which yields
\begin{equation*}\begin{split}
2K
&\leq{K}+K(2K)^{\frac{n+2\gamma}{n-2\gamma}
    -\frac{2n}{n-2\gamma}\frac1q}\left(\int_{\Gamma_{2r}^{0}(x)}|R_g|^q\,d\mu_g\right)^{\frac1q}\\
&\leq
    K+K(2K)^{\frac{n+2\gamma}{n-2\gamma}-\frac{2n}{n-2\gamma}\frac1q}\eta_0^{\frac1q}.
\end{split}\end{equation*}
Clearly this contradicts the choice of $\eta_0$. Thus we must have
\[
\left(\frac{r-r_0}{2}\right)^{\frac{n-2\gamma}{2}}U(x_0)
\leq 2K.
\]
Using \eqref{eq:190.3} with $s$ replaced by $\frac{r-r_0}{2}$, we obtain
\begin{equation*}
\begin{split}
&\quad\;\left(\frac{r-r_0}{2}\right)^{\frac{n-2\gamma}{2}}U(x_0)\\
&\leq{K}\left(\frac{r-r_0}{2}\right)^{\frac{n-2\gamma}{2}}\\
&\quad\;
    +K(2K)^{\frac{4\gamma}{n-2\gamma}-\frac{2n}{n-2\gamma}\frac1q}
    \left(\int_{\Gamma_{2r}^{0}(x)}|R_g|^q\,d\mu_{g}\right)^{\frac1q}
    \left(\frac{r-r_0}{2}\right)^{\frac{n-2\gamma}{2}}U(x_0).
\end{split}
\end{equation*}
Since $\left(\int_{\Gamma_{2r}^{0}(x)}|R_g|^q\,d\mu_g\right)^{\frac1q}\leq \eta_0^{\frac{1}{q}}$ and $(2K)^{\frac{n+2\gamma}{n-2\gamma}-\frac{2n}{n-2\gamma}\frac1q}\eta_0^{\frac1q}\leq \frac12$, then
\begin{align*}
\left(\frac{r-r_0}{2}\right)^{\frac{n-2\gamma}{2}}U(x_0)
&\leq2K\left(\frac{r-r_0}{2}\right)^{\frac{n-2\gamma}{2}}
\leq2Kr^{\frac{n-2\gamma}{2}}.
\end{align*}
Thus we conclude that
\begin{align*}
r^{\frac{n-2\gamma}{2}}U(x)
&\leq 
    (r-r_0)^{\frac{n-2\gamma}{2}}U(x_0)\\
&\leq2^{\frac{n+2-2\gamma}{2}}Kr^{\frac{n-2\gamma}{2}},
\end{align*}
as desired.
\end{proof}

\bibliography{mybibfile}

\end{document}